\documentclass[11pt]{amsart}
\usepackage{amsmath,amssymb,enumerate}
\newtheorem{prop}{Proposition}[section]
\newtheorem{coro}[prop]{Corollary}

\newtheorem{lem}[prop]{Lemma}
\newtheorem{rem}[prop]{Remark}
\newtheorem{exa}[prop]{Example}

\newtheorem{defi}[prop]{Definition}
\newtheorem{theo}[prop]{Theorem}

\newcommand{\B}{\mathbb B}

\newcommand{\R}{\mathbb R}

\newcommand{\C}{\mathbb C}
\newcommand{\N}{\mathbb N}

\newcommand{\e}{\varepsilon}
\newcommand{\f}{\varphi}

\newcommand{\p}{\psi}

 \newcommand \al {\alpha}
 \newcommand \de {\delta}

 \newcommand \Om {\Omega}

 \numberwithin{equation}{section}

\newenvironment{ack}
{\vskip .2cm \noindent {{\bf Acknowledgements.}}}{\hfill\break}

\begin{document}

  \title[Degenerate complex Monge-Amp\`ere flows]{Weak solutions to degenerate complex Monge-Amp\`ere flows I}

\setcounter{tocdepth}{1}

  \author{ Philippe Eyssidieux, Vincent Guedj, Ahmed Zeriahi} 

\address{Universit\'e Joseph Fourier et Institut Universitaire de France, France}

\email{Philippe.Eyssidieux@ujf-grenoble.fr }

\address{Institut de Math\'ematiques de Toulouse et Institut Universitaire de France,   \\ Universit\'e Paul Sabatier \\
118 route de Narbonne \\
F-31062 Toulouse cedex 09\\}

\email{vincent.guedj@math.univ-toulouse.fr}

\address{Institut de Math\'ematiques de Toulouse,   \\ Universit\'e Paul Sabatier \\
118 route de Narbonne \\
F-31062 Toulouse cedex 09\\}

\email{zeriahi@math.univ-toulouse.fr}

 \date{\today}

\begin{abstract}  
Studying the (long-term) behavior of the K\"ahler-Ricci flow on mildly singular varieties, one is naturally lead to study weak solutions of degenerate parabolic complex Monge-Amp\`ere equations.

  The purpose of this article, the first of a series on this subject, is to develop a viscosity theory for
degenerate complex Monge-Amp\`ere flows in domains of $\C^n$. 
\end{abstract}

 \maketitle

\tableofcontents

\newpage

\section*{Introduction}
  
  The study of the (long-term) behavior of the K\"ahler-Ricci flow on mildly singular varieties in relation to the Minimal Model Program
was undertaken by J. Song and G. Tian \cite{ST1,ST2} and requires 
a theory of weak solutions for certain  degenerate parabolic complex Monge-Amp\`ere equations modelled on:
\begin{equation}\label{krf}
 \frac{\partial \phi}{\partial t}+\phi= \log \frac{(dd^c\phi)^n}{V}
\end{equation}
where $V$ is volume form and $\phi$ a $t$-dependant K\"ahler potential on a compact K\"ahler manifold. The approach in \cite{ST2} is to regularize the equation
and take limits of the solutions of the regularized equation with uniform
higher order estimates. But as far as the existence and uniqueness statements in \cite{ST2} are concerned, we believe that 
a zeroth order approach would be both simpler and more efficient.

There is a well established pluripotential 
theory of weak solutions to elliptic complex Monge-Amp\`ere equations,
following the pionnering work of Bedford and Taylor \cite{BT76,BT82} in the local case (domains in $\C^n$). 
A complementary viscosity approach has been developed only recently in \cite{HL09,EGZ11,W12, EGZ13} both in the local and the global case (compact K\"ahler manifolds).

Suprisingly no similar theory has ever been developed on the parabolic side.  The most significant reference for a parabolic flow of plurisubharmonic functions on pseudoconvex domains
 is \cite{Gav} but the  flow studied there takes  the form 
\begin{equation}
 \label{gavflow}  \frac{\partial \phi}{\partial t}= ((dd^c\phi)^n )^{1/n} 
\end{equation}
which does not make sense in the global case. 
The purpose of this article, the first of a series on this subject, is to develop a viscosity theory for
degenerate complex Monge-Amp\`ere flows of the form 
(\ref{krf}).

This article focuses on solving this problem in domains of $\C^n$, while its companion \cite{EGZ14}
is concerned with the global case.
More precisely  we study here the degenerate parabolic complex Monge-Amp\` ere equations
\begin{equation} \label{krflocal}
 e^{\partial_t{\f}+  F(t,z,\f_t)} \, \mu(z)-(dd^c \f_t)^n=0 \; \; \text{ in } \; \; \Omega_T
\end{equation}

  where
  \begin{itemize}
  \item $\Omega \Subset \C^ n$ is a smooth bounded strongly pseudoconvex domain, 
 \item $T \in ]0,+\infty]$;
  \item $F (t,z,r)$ is continuous in $[0,T[ \times \Omega \times \R$ and non decreasing in $r$,
  \item $\mu (z)\geq 0$ is a bounded continuous volume form on $\Omega$,
   \item $\f : \Omega_T :=  [0,T[ \times \Omega \rightarrow \R$ is the unknown function, with $\f_t: = \f (t,\cdot)$. 
   \end{itemize}
Our plan is to adapt the viscosity approach developped by P.L.~Lions and al.
(see \cite{IL90,CIL92}) to the complex case, using the elliptic side of the theory which was developped in \cite{EGZ11}. It should be noted that the
method used in \cite{ST2} is a version of the classical PDE method of vanishing viscosity which was superseded by the theory of viscosity solutions. 

\smallskip

After developing the appropriate definitions of (viscosity) subsolution, supersolution and solution
in the {\it first section}, we establish 
in the {\it second section} an important connection with the elliptic side of the theory:

\medskip

\noindent{\bf{Theorem A.}} 
{\em  
 If $u$ is a bounded subsolution of the above degenerate parabolic complex Monge-Amp\`ere flow (\ref{krflocal}) in $]0,T[ \times \Omega$, then 
$z \mapsto u(t,z)$ is plurisubharmonic in $\Omega$ for all $t > 0$.
}

\medskip

As is often the case in the viscosity theory, one of our main technical tools is the following comparison principle, which we establish in the {\it third section}:

\medskip

\noindent{\bf{Theorem B.}} 
{\em  
If $u$ (resp. $v$) is a bounded subsolution (resp. supersolution) 
of the above degenerate parabolic equation
then
$$
\max_{\Omega_T} (u-v) \leq \max\{0, \max_{\partial_0 \Omega_T} (u-v)\}.
$$
}

\medskip

Here $\partial_0 \Omega_T=\left(\{0\} \times \overline{\Omega}\right) \cup \left([0, T[ \times \partial \Omega\right)$ denotes the parabolic boundary of $\Omega_T$.
We actually establish several variants of the comparison principle
 (see Theorem \ref{thm:comparison}  and the remarks following its proof).

\smallskip

In the {\it fourth section} we construct barriers at each point of the parabolic boundary and use the Perron method to eventually show the existence of a viscosity solution to the Cauchy-Dirichlet problem for the Complex Monge-Amp\`ere flow (\ref{krflocal}) (see Section 1):

\medskip

\noindent{\bf{Theorem C.}} 
{\em  
 Let  $\f_0$ be a continuous plurisubharmonic function on $\overline{\Omega}$ 
 such that $(\f_0,\mu)$ is admissible in the sense of definition \ref{def:adm}. 
 
The Cauchy-Dirichlet problem for the parabolic complex Monge-Amp\`ere equation  
with initial data $\f_0$ admits a unique viscosity solution $\f (t,x)$ in infinite time;
it is the upper envelope of all subsolutions. 
}

\medskip

We give simple criteria in Lemma \ref{lem:admissible} to decide whether a data
$(\f_0,\mu)$ is admissible. This is notably always the case when $\mu>0$ is positive,
while we can not expect the existence of a supersolution if $\mu$ vanishes and
$\f_0$ is not a maximal plurisubharmonic function.

We finally study the long term behavior of the flow in  {\it section five}, showing
that it asymptotically recovers the solution of the corresponding elliptic Dirichlet
problem (see Theorems  \ref{thm:conv<0} and \ref{thm:conv}):

\medskip

\noindent{\bf{Theorem D.}} 
{\em  
Assume $(\f_0,\mu)$ is admissible and $F=F(z,r)$ is time independent.
 The complex Monge-Amp\` ere flow $\f_t$ starting at $\f_0$ uniformly converges, 
as $t \rightarrow +\infty$, to the solution $\p$
of the Dirichlet problem for the degenerate elliptic Monge-Amp\`ere equation
$$
(dd^c \p)^n=e^{F(z,\p)} \mu(z) \text{ in } \Omega,
\; \; \text{ with }
\p_{| \partial \Omega}=\f_0.
$$
}

\medskip

The solution $\p$ to the above elliptic Dirichlet problem is well known to exist  
in the pluripotential sense \cite{Ceg84}, while its existence in the viscosity sense
was established in \cite{HL09,EGZ11,W12}.

\medskip
Pluripotential theory actually suggests that the solutions to (\ref{krflocal}) should be defined 
as upper semi continuous $t$-dependant plurisubharmonic functions which are a.e. derivable w.r.t to 
the time variable and satisfy the equation almost everywhere where $(dd^c\phi)^n$ is replaced by the Monge-Amp\`ere operator. We did not try and phrase 
such a definition in a precise and usable way nor determine how it connects to the viscosity concepts developped here. 
\medskip
\begin{ack} 
We thank Cyril Imbert for useful discussions. 
\end{ack}

\section{Parabolic viscosity concepts}

\subsection{A Cauchy-Dirichlet problem}
Let $\Omega \Subset \C^n$ be  a bounded strongly pseudoconvex domain and $T > 0$ a fixed number and define
$$
\Omega_T := ]0,T[ \times \Omega.
$$
We are studying the parabolic complex Monge-Amp\`ere  equation (\ref{eq:PMF}) 
 \begin{equation} \label{eq:PMF} 
e^{\partial_t{\f}+  F(t,z,\f)} \, \mu(t,z)-(dd^c \f_t)^n=0 \; \; \text{ in } \; \; \Omega_T,
  \end{equation}
where $F (t,z,r)$ is continuous in $[0,T[ \times \Omega \times \R$ and non decreasing in $r$.
The measure $\mu=\mu(t,z) = \mu_t (z) \geq 0$ is assumed to be a bounded continuous non negative volume form depending continuously on the time variable $t$.
It will be often necessary to also impose further that either $\mu>0$ is positive, or that
$\mu=f(z) \nu(t,z)$ with $f(z) \geq 0$ and $\nu=\nu(t,z)>0$. The positive part of the
density can then be absorbed in $F$. For simplicity, we have therefore stated our main results in the introduction in the case when $\mu=\mu(z) \geq 0$ is time independent
but will use  a slightly larger framework in the bulk of the article.

We call this equation the parabolic Monge-Amp\`ere equation associated to $(F,\mu)$ in $\Omega_T$.

 Recall that the parabolic boundary of $\Omega_T$ is defined as the set
 $$
 \partial_0 \Omega_T := \left(\{0\} \times \overline{\Omega}\right) \cup \left([0, T[ \times \partial \Omega\right).
 $$ 
  We want to study the Cauchy-Dirichlet problem for (\ref{eq:PMF})  with the following Cauchy-Dirichlet conditions:
 \begin{eqnarray} \label{eq:CD-PMF}
\left\{ \begin{array}{rcl}
  \f (0,z) &=& \f_0 (z), \, \, \, \, (0,z) \in \{0\} \times \overline{\Omega}, \\
  \f (t,\zeta) & = & h (\zeta), \ \, \,  (t,\zeta) \in [0,T[ \times {\partial \Omega},
\end{array}
\right.
 \end{eqnarray}
 where $h : \partial \Omega \rightarrow \R$ is a continuous function (the Dirichlet boundary data) and $\f_0$ is a bounded plurisubharmonic  function in $\Omega$  (the Cauchy data), which extends continuously
to $\overline{\Omega}$.

Thus $h$ is actually determined by the boundary values of $\f_0$.  
Such a function $\f_0$ will be called the {\it Cauchy-Dirichlet data} for the parabolic complex Monge-Amp\`ere equation  (\ref{eq:PMF}) and we will simply write
$$
\f_{| \partial \Omega_T}=\f_0.
$$

\subsection{Parabolic sub/super-solutions}

We assume the reader has some familiarity with the elliptic side of the viscosity theory
for complex Monge-Amp\`ere equations which was developed in \cite{EGZ11}.

The definitions of subsolutions and supersolutions can be extended to the parabolic setting using upper and lower test functions as in the degenerate elliptic case.

We first define what should be a classical solution to our problem.  A classical solution to the parabolic complex Monge-Amp\`ere equation (\ref{eq:PMF}) is a continuous function $ \f : [0,T[ \times \overline{\Omega} \longrightarrow \R$ which is $C^1$ in $t$, $C^2$ in $z$  in 
$]0,T[ \times \Omega$ such that for any $t \in ]0,T[$, the function $z \longmapsto \f (t,z)$ is a (continous) plurisubharmonic function in $\Omega$ that satisfies the following equation 
$$
(dd^c \f_t)^n = e^{\partial_t \f (t,z) + F (t,z,\f (t,z))} \mu (t,z),
$$
for all $ z \in \Omega$. The function $\f$ is said to be  $C^1$ in $t$ and $C^2$ in $z$ (or $C^{(1,2)}$ in short)  in $ ]0,T[ \times \Omega$ if $(t,z) \longrightarrow \partial_t \f (t,x)$ exists and is continuous in $ ]0,T[ \times \Omega$ and the second partial derivatives  of $z \longrightarrow \f (t,z)$ with respect to $z_j$ and $\bar z_k$ exists and are continuous in all the variables $(t,z)$ in $]0,T[ \times \Omega$.

Observe that if we split this equality into two inequalities $\geq $ (resp. $\leq $), we obtain the notion of a classical subsolution (resp. supersolution) to the parabolic equation (\ref{eq:PMF}).

Now let us introduce the general definition.
\begin{defi} (Test functions)
 Let  $ w : \Omega_T \longrightarrow \R$ be any function defined in $\Omega_T$ and $(t_0,z_0) \in \Omega_T$ a given point. 
  An upper test function (resp. a lower test function) for $w$ at the point $(t_0,z_0)$ 
is  a $C^{(1,2)}$-smooth function $q$ in a neighbourhood of the point $(t_0,z_0)$ such that $  w (t_0,z_0) = q (t_0,z_0)$ and $w \leq q$ (resp. $w \geq q$) in a neighbourhood of $(t_0,z_0)$.
 We will write for short $w \leq_{(t_0,z_0)} q$ (resp.  $w \geq_{(t_0,z_0)} q$).  
 \end{defi}
 
 \begin{defi}
 1. A  function  $ u : [0,T[ \times \overline{\Omega} \longrightarrow \R$ is said to be a (viscosity) subsolution to the parabolic complex Monge-Amp\`ere equation (\ref{eq:PMF}) in $]0,T[ \times \Omega$ if $u$ is upper semi-continuous in $[0,T[ \times \overline{\Omega}$ and for any point $ (t_0,z_0) \in \Omega_T := ]0,T[ \times \Omega$ and any upper test function $q$ for $u$ at $(t_0,z_0)$, we have
 $$
 (dd^c q_{t_0} (z_0))^n  \geq e^{\partial_t q (t_0,z_0) + F (t_0,z_0,q (t_0,z_0))} \mu (t_0,z_0).
 $$
 In this case we also say that $u$ satisfies the differential inequality $(dd^c \f_t)^n \geq e^{\partial_t \f (t,z) + F (t,z,\f (t,z))} \mu (t,z)$ in the viscosity sense in $\Omega_T$.
 
 2. A  function  $ v : [0,T[ \times \overline{\Omega} \longrightarrow \R$ is said to be a (viscosity) 
supersolution to the parabolic complex Monge-Amp\`ere equation (\ref{eq:PMF}) 
in $\Omega_T = ]0,T[ \times \Omega$ if $v$ is lower semi-continuous in $\Omega_T$ and for any point 
$ (t_0,z_0) \in ]0,T[ \times \Omega$ and any lower test  function $q$ for $v$ at $(t_0,z_0)$ such that $dd^c q_{t_0}(z_0) \geq 0$, we have 
 $$
 (dd^c q_{t_0})^n (z_0) \leq e^{\partial_{t} q (t_0,z_0) + F (t_0,z_0,q (t_0,z_0))} \mu (t_0,z_0).
 $$
  In this case we also say that $v$ satisfies the differential inequality $(dd^c \f_t)^n \leq e^{\partial_t \f (t,z) + F (t,z,\f (t,z))} \mu (t,z)$ in the viscosity sense in $\Omega_T$.
  
 3. A function $ \f : [0,T[ \times \overline{\Omega} \longrightarrow \R$ is said to be a (viscosity) solution to the parabolic complex Monge-Amp\`ere equation (\ref{eq:PMF}) in $]0,T[ \times \Omega$ if it is a subsolution and a super solution to the parabolic complex Monge-Amp\`ere equation (\ref{eq:PMF}) in $]0,T[ \times \Omega$. Hence $\f$ is continuous in $[0,T[ \times \overline{\Omega}$.
  \end{defi}

We let the reader check that a classical (sub/super) solution of equation  (\ref{eq:PMF})  is a viscosity
(sub/super) solution.

\begin{rem} \label{rem:parab=ellipt}
In order to fit into the framework of viscosity theory, we consider the function
$H : [0,T[ \times \Omega \times \R \times \R \times \R^n \times \mathcal S_{2 n} \longrightarrow \R \cup \{+ \infty\}$ defined by
$$
\left\{ \begin{array}{rcl}
  H (t,z,r,\tau,p,Q) &:=&  e^{\tau + F (t,z,r)} \mu(t,z) - (dd^c Q)^n, \, \, \, \, \text{if} \, \, dd^c Q \geq 0, \\
  H (t,z,r,\tau,p,Q) & := &+ \infty, \ \, \,                                        \text{if not},
\end{array}
\right.,
$$
where $dd^c Q$ is the hermitian $(1,1)$-part of the quadratic form $Q$ in $\C^ n \simeq \R^ {2 n}$.

Observe that the  function $H$ is lower semi-continuous in the set $[0,T[ \times \Omega \times \R \times \R \times \R^n \times \mathcal S_{2 n}$, continuous in its domain  
$$
Dom H := \{H < + \infty\} = [0,T[ \times \Omega \times \R \times \R \times \R^n \times \{Q \in \mathcal S_{2 n} ; dd^c Q \geq 0\}
$$ 
and is degenerate elliptic in the sense of \cite{CIL92}. Moreover it is non decreasing in the $r$ variable. We call it the {\it Hamilton function} of the parabolic complex Monge-Amp\`ere equation (\ref{eq:PMF}).

Observe that if $u $ is a subsolution (resp. a supersolution) of the parabolic equation $H = 0$ then it is a subsolution of the  degenerate elliptic equation
$H = 0$ in $2 n + 1$ variables $(t,z) \in ]0,T[ \times \Omega \subset \R^ {2 n + 1}$ of a special type which does not depend on the gradient w.r.t. $z$ nor on the second derivative w.r.t. $t$. Actually the two notions are equivalent  but we will not use this (see \cite{CIL92}).

The notions of subsolutions and supersolutions  for the parabolic equation $H = 0$ as defined in 
\cite{CIL92} are exactly the one defined above.
 
However as far as supersolutions are concerned, it is more useful to work with the  finite Hamilton function $H_+$, where
$$
  H_+(t,z,r,\tau,p,Q) :=  e^{\tau + F (t,z,r)} \mu(t,z) - (dd^c Q)_+^n,                                  
$$
and $(dd^c Q)_+ = dd^ c Q$ if $dd^ c Q \geq 0$ and $(dd^c Q)_+ = 0$ if not.

Observe that $H_+ : [0,T[ \times \Omega \times \R \times \R \times \R^n \times \mathcal S_{2 n} \longrightarrow \R$ is an upper semi-continuous and finite Hamilton function such that $H_+ = H$ in $Dom H$, the domain of $H$ .

Therefore most of the general principles of the viscosity method as explained in \cite{CIL92} can be also applied here (at least formally).
On the other hand we have to be careful since there is no symmetry between subsolutions and supersolutions.

It follows from \cite{EGZ11} that if $u$ is a subsolution to the parabolic equation $H = 0$, any parabolic upper test function $q$ for $u$ at $(t_0,z_0)$ satisfies the condition $dd^ c q_{t_0} (z_0) \geq 0$. 
Hence $u$ is a subsolution to the parabolic equation $H_+ = 0$, but the converse is not true unless $\mu > 0$ (see  \cite{EGZ11}).
\end{rem}

Since the fundamental Jensen-Ishii's maximum principle will be stated in terms of semi-jets, it is 
 convenient to use these notions which we now introduce following \cite{CIL92},
in order to characterize as well the notions of sub/super solutions.

\smallskip

\begin{defi}  Let  $ u : \Omega_T \longrightarrow \R$ be a fixed function.  
For $(t_0,z_0) \in \Omega_T$, the parabolic second order superjet of $u$ at $(t_0,z_0)$ is the set of $(\tau,p,Q) \in \R \times \R^{2 n} \times \mathcal S_{2 n} $ such that for $(t,z) \in \Omega_T$, 
\begin{eqnarray*} \label{eq:superjet} 
  u (t,z) & \leq & u (t_0,z_0) +  \tau (t-t_0) + o (\vert t-t_0\vert) \\
  &+&  \nonumber  \langle p, z - z_0\rangle + \frac{1}{2} \langle Q (z - z_0), z-z_0\rangle   
+ o (\vert z - z_0\vert^2). 
\end{eqnarray*}
 We let $\mathcal P^{2,+} u (t_0,z_0)$ denote the set of parabolic second order superjets of $u$ at $(t_0,z_0)$.
We define in the same way  the  set  $\mathcal P^{2,-} u (t_0,z_0)$ 
of  parabolic second order subjets of $u$ at $(t_0,z_0)$ by 
$$
\mathcal P^{2,-} u (t_0,z_0) = - \mathcal P^{2,+} (-u) (t_0,z_0).
$$  
 The set of parabolic second order jets of $u$ at $(t_0,z_0)$ is defined by
$$
\mathcal P^{2} u (t_0,z_0) = \mathcal P^{2,+} u (t_0,z_0) \cap \mathcal P^{2,-} u (t_0,z_0).
$$
 \end{defi}
 
 We will need a slightly more general notion (see \cite{CIL92}).   
 
 \begin{defi}\label{def:closedparasuperjet}  
Let  $ u : \Omega_T \longrightarrow \R$ be a fixed function and  $(t_0,z_0) \in \Omega_T$. The set   $\bar{\mathcal P}^{2,+} u (t_0,z_0)$ of approximate parabolic second order superjets of $u$ at $(t_0,z_0)$  is defined as the set  of $(\tau,p,Q) \in \R \times \R^{2 n} \times \mathcal S_{2 n}$ such that there exists  a sequence $(t_j,z_j) \in ]0,T[ \times \Omega$ converging to $(t_0,z_0)$, such that $ u (t_j,z_j) \to u (t_0,z_0)$ and a sequence $(\tau_j,p_j,Q_j) \in \mathcal P^{2,+} u (t_j,z_j)$  converging to $(\tau,p,Q)$. 

 In the same way we define the set $\bar{\mathcal P}^{2,-} u (t_0,z_0) : = - \bar{\mathcal P}^{2,+} (-u) (t_0,z_0)$ of approximate parabolic second order subjets of $u$ at $(t_0,z_0)$.
\end{defi}

 \begin{prop} 
\text{ }

1. An upper semi-continuous function $u : \Omega_T \longrightarrow \R$ is a subsolution to the parabolic equation (\ref{eq:PMF}) if and only if for all
$ (t_0,z_0) \in \Omega_T$ and  $(\tau,p,Q) \in {\mathcal P}^{2,+} u (t_0,z_0),$  we have
 \begin{equation} \label{eq:sub}
  e^ {\tau +  F (t_0,z_0, u (t_0,z_0))} \mu (t_0,z_0)  \leq (dd^c Q)^n.
 \end{equation}

 2. A lower semi-continuous function $v : \Omega_T \longrightarrow \R$  is 
a supersolution to the parabolic equation (\ref{eq:PMF})  if and only if  for all
$ (t_0,z_0) \in \Omega_T$ and  $(\tau,p,Q) \in {\mathcal P}^{2,-} v (t_0,z_0)$ such that $dd^c Q \geq 0,$ we have
 \begin{equation} \label{eq:super}
  e^ {\tau +  F (t_0, z_0,v (t_0,z_0))} \mu (t_0,z_0)   \geq (dd^c Q)^n.
 \end{equation}
 \end{prop}

Another way to phrase the definition of supersolutions is to require that,  for all
$ (t_0,z_0) \in \Omega_T$  and all $(\tau,p,Q) \in {\mathcal P}^{2,-} v (t_0,z_0)$, we have
$$
  e^ {\tau +  F (t_0, z_0,v (t_0,z_0))} \mu (t_0,z_0)   \geq (dd^c Q)_+^n.
$$

This statement necessitates some comments:
 \begin{enumerate}
\item The reader will easily check that when $u$ is a subsolution (resp. supersolution) to the equation (\ref{eq:PMF}), the inequalities (\ref{eq:sub}) (resp. \ref{eq:super}) are satisfied for all $(\tau,p,Q) \in \overline{\mathcal P}^{2,+} u (t_0,z_0)$ (resp. $\overline{\mathcal P}^{2,-} v (t_0,z_0)).$  
\item If for  a fixed $ z_0 \in \Omega$, the function $t \longmapsto u (t,z_0)$ is $L$-lipschitz 
 in a neighborhood of $t_0 \in ]0,T[$. Then for any $(\tau,p,Q) \in \overline{\mathcal P}^{2,+} u (t_0,z_0)$, we have $\vert \tau \vert \leq L$.
 Indeed for $\vert s \vert << 1$ and $\vert z - z_0\vert << 1$, 
 \begin{eqnarray*}
 u (t_0 + s,z) &\leq& u (t_0,z_0) + \tau s + \langle p,z - z_0\rangle \\
 &  +&  \frac{1}{2} \langle Q (z-z_0),z - z_0\rangle + o (\vert s\vert+\vert z - z_0\vert^2),
 \end{eqnarray*}
  hence $- L \vert s\vert \leq u (t_0 + s,z_0) - u (t_0,z_0)  \leq \tau s + o (\vert s\vert)$ for  $\vert s\vert $ small enough and the conclusion follows.
  \item A discontinuous viscosity solution to the equation (\ref{eq:PMF}) (in the the sense of
 \cite{Ish89}) is a function $u:\Omega_T \to [+\infty, -\infty]$ such that

i) the usc envelope $u^*$ of $u$ satisfies $\forall x, \ u^*(x)<+\infty$ and is a viscosity subsolution to (the equation (\ref{eq:PMF}),

ii) the lsc envelope $u_*$ of $u$ satisfies $\forall x, \ u_*(x)>-\infty$ and is a viscosity supersolution to the equation (\ref{eq:PMF}).
 \end{enumerate}
 
 If we consider a time independent equation, its static viscosity (sub/super) solutions (i.e.: those who are independent of the time variable) are the time independent extension of the viscosity (sub/super) solutions of the corresponding Complex Monge-Amp\`ere equation in the sense of \cite{EGZ11} where discontinuous viscosity solutions were not considered.

We introduce discontinuous viscosity solutions here for technical reasons that will be explained later on. Note that the characteristic function $u$ of $\mathbb{C}\backslash\mathbb{Q}^2$ is a discontinuous viscosity solution of  $\Delta u=0$.

\subsection{Relaxed semi-limits}
 Let $(h_j)$ be a sequence of locally uniformly bounded functions on a metric space $(Y,d)$. The upper relaxed semi-limit of  $(h_j)$ is
$$
\overline h (y) = {\limsup}^*_{j \to + \infty} h_j (x) := \lim_{j \to + \infty} \sup \{ h_k (z) ; k \geq j,  d (z,y) \leq 1 \slash j\}. 
$$
The reader will easily check that $\overline h $ is upper semi-continuous on $Y$.

We define similarly the lower relaxed semi-limit of the sequence $(h_j)$,
$$
\underline h = {\liminf}_*{_{j \to + \infty}} h_j.
$$ 
This is a lower semi-continuous function in $Y$.
Observe that 
$$
 {\liminf}_*{_{j \to + \infty} h_j} \leq (\liminf_{j \to \infty} h_j)_* \leq (\limsup_{j \to + \infty} h_j)^* \leq {\limsup}^*_{j \to + \infty} h_j.
$$

If  $(h_j)$ is a non decreasing (resp. non increasing) sequence of continuous functions on $Y$ then $\overline h = (\sup h_j)^*$
(resp. $\underline h = (\inf h_j)_*)$. Moreover if  $(h_j)$ converges locally uniformly to a continuous function $h$ on $Y$ then all these limits coincide with $h$ on $Y$.

The following stability result for viscosity sub/super-solutions is a classical and useful tool  
(see \cite{CIL92,IS12}):

\begin{lem} \label{lem:Stab}
Let $\mu^j (t,x) \geq 0$ be a sequence of continuous volume forms converging uniformly to a volume form $\mu$ on $\Omega_T$ and let $F^j$ be a sequence of continuous functions in $[0,T[ \times \Omega \times \R$ converging locally uniformly to a function $F$.
Let $(\f^j)$ be a   locally uniformly bounded sequence of real valued functions defined in $\Omega_T$. 

1. Assume that for every $j \in \N$, $\f^j$ is a viscosity subsolution to the complex Monge-Amp\` ere flow
$$
 e^{ \partial_t \f^j + F^j (t,z,\f^j)} \mu^j (t,z) - (dd^c \f^j_t)^n = 0,
$$
associated to $(F^j,\mu^j)$ in $\Omega_T$.
Then  its upper relaxed semi-limit 
$$
\overline \f = {\limsup}^*_{j \to + \infty} \f^j
$$
of the sequence $(\f_j)$ is a subsolution to the parabolic Monge-Amp\` ere equation
$$
 e^{ \partial_t \f + F (t,z,\f)} \mu - (dd^c \f_t)^n = 0, 
$$
in $\Omega_T$. 

2. Assume that for every $j \in \N$, $\f^j$ is a viscosity supersolution to the complex Monge-Amp\`ere flow associated to $(F^j,\mu^j)$ in $\Omega_T$. Then the lower relaxed semi-limit 
$$
\underline \f = {{\liminf}_*}_{j \to + \infty} \f^j
$$ 
of the sequence  $(\f^j)$ is a supersolution to the complex Monge-Amp\`ere flow associated to $(F,\mu)$ in $\Omega_T$.
\end{lem}

It is a remarkable fact that we do not need any a priori estimate on the time derivatives to pass to the limit in the viscosity differential inequalities.

\begin{rem}
 An important example in applications is when $F (t,z,r) = \alpha r$ with $\alpha \geq 0$. In this case a simple change of variables reduces the general case $\alpha > 0$ to the case when $\alpha =0$.
 Indeed if $\alpha > 0$ set
 $$
 \psi (s,z) := \alpha (1 + s) \f (t,z), \, \, \text{with} \, \,  t := \alpha^{-1} \log (1 + s),
 $$
 and observe that
 $$
 \partial_s \psi (s,z) = \alpha \f (t,z) + \partial_t \f (t,z).
 $$
 Thus $\f$ is a (sub/super)solution to the the parabolic Monge-Amp\` ere equation 
 $$
 \exp \left ( \partial_t \f + \al \f \right) \mu - (dd^c \f_t)^n = 0, 
 $$ 
 if and only $\psi$ is a (sub/super)solution to 
 $$
 e^{\partial_t \p}  \tilde \mu (s,\cdot) - (dd^c \p_s)^n = 0,
 $$
 where $\tilde \mu (s,z) = \alpha^ n (s +1)^n \mu (z).$
 \end{rem}

\section{The parabolic Jensen-Ishii's maximum principle}

\subsection{Maximum principles}

 Recall that a function $u : U \subset \R^N \longrightarrow \R$ is  semi-convex in $U$ if for each small ball $B \Subset U$, there exists a constant $A > 0$ such that the function $x \longmapsto u (x) + A \vert x\vert^2$ is convex in $B$.

We also recall that the upper second order jet  ${\mathcal J}^{2,+} u (x_0)$ at 
  $x_0 \in U$ of a function $u : U \longrightarrow \R$ is the set of $(p,Q) \in \R^N \times \mathcal S_N$ s.t. for $x$ close to $x_0$,
  $$
 u (x) \leq u (x_0) + \langle p,x - x_0\rangle + \frac{1}{2} \langle Q (x-x_0),x - x_0\rangle +  o (\vert x - x_0\vert^2).
 $$
 The set $\bar{\mathcal J}^{2,+} u (x_0)$ of approximate second order superjets 
is then defined in the same way as in Definition \ref{def:closedparasuperjet}. 

The following is a consequence of the fundamental Theorem of Jensen on which the Jensen-Ishii maximum principle is based (see \cite{CIL92}, \cite{Car04}):
 
  \begin{theo} 
Let $u$ be a semi-convex function in an open set $U \subset \R^N$,  attaining a local maximum
at some point $x_0 \in U$. Then there exists $(p,Q) \in  \bar{\mathcal J}^{2,+} u (x_0)$ such that 
$p = 0$ and $Q \leq 0$. 

More precisely for any subset $E \subset U$ of Lebesgue measure $0$, there exists a sequence $(x_k)$ of points in $U\setminus E$ such that $x_k \to x_0$, $u (x_k) \to u (x_0)$, $u$ is twice differentiable at each point $x_k$  for $k > 1$, $D u (x_k) \to 0$ and $D^2 u (x_k) \to Q \leq 0$ as $k \to + \infty$. 
  \end{theo}

A crucial ingredient is Alexandrov's theorem on almost everywhere second order differentiability 
of convex functions \cite{Ale39}. From this we can derive the following useful result:
 
\begin{lem} \label{lem:subpp}
 Let $U \subset \R^N$ be an open set and  $H : U \times \R \times \R^N \times \mathcal S_N \longrightarrow \R \cup \{+ \infty\}$ be a function. 
 
1. Assume that $H$ is  lower semi-continuous and  degenerate elliptic. Let $w$ be a semi-convex function in the open set $U \subset \R^N$ such that for almost all $x_0 \in U$, 
$$
H (x_0, w (x_0),p,Q) \leq 0, \forall (p,Q) \in  {\mathcal J}^{2,+} w (x_0). 
$$
Then $w$ is a (viscosity) subsolution to the equation $H = 0$ in $U$. 

2. Assume that $H : U \times \R \times \R^N \times \mathcal S_N \longrightarrow \R$ is finite, upper semi-continuous and  degenerate elliptic. Let $w$ be a semi-concave function in the open set $U \subset \R^N$ such that for almost all $x_0 \in U$, 
$$
H (x_0, w (x_0),p,Q) \geq 0, \forall (p,Q) \in  {\mathcal J}^{2,-} w (x_0). 
$$
Then $w$ is a (viscosity) supersolution to the equation $H = 0$ in $U$.
\end{lem}

In other words, if $w$ is a subsolution (resp. supersolution) of the equation $H = 0$ almost everywhere in $U$ , then it is a subsolution (resp. supersolution) everywhere.

\begin{proof} We prove the first statement concerning subsolutions. The second statement concerning supersolutions can be proved in the same way.

 We will use the maximum principle of Jensen for semi-convex functions.
Let $q$ be a $C^2$ upper test function for $w$ at a fixed point $x_0$.
Thus $u := w - q$ is semi-convex function in $U$ which takes a local maximum at $x_0$. 

Let us denote by $E$ the exceptional set of points where the viscosity inequality in the lemma is not satisfied. Since $E$ has Lebesgue measure $0$, it follows from the local maximum principle of Jensen for the semi-convex function $u$ that there exists a sequence $x_k$ in $U \setminus E$ converging to $x_0$ such that $u$ is twice differentiable at $x_k$, $D u (x_k) \to 0$ and $D^2  u (x_k) \to A \leq 0$ as $k \to + \infty$ i.e. $(0,A) \in \bar{\mathcal J}^{2,+} u (x_0)$ and $A \leq 0$. 

 By definition $ D u (x_k) = D w (x_k) - D q (x_k)$ and $D^2 u (x_k) = D^2 w (x_k) - D^2 q (x_k)$ for any $k$ hence $p_k := D w (x_k) = D u (x_k) + D q (x_k) \to D q (x_0)$ and $Q_k := D^2 w (x_k) \to A + D^2 q (x_0) =: Q$.
Therefore  $Q \leq D^2 q (x_0)$ since $A \leq 0$.

By the choice of $x_k$, we infer $H (x_k, p_k,Q_k) \leq 0$. By the lower semi-continuity of $H$ we get at the limit
$ H (x_0, D q (x_0), Q) \leq 0$. Since $Q \leq D^2 q (x_0)$, by the degenerate ellipticity condition, we conclude that 
$$
H (x_0, D q (x_0), D^2 q (x_0)) \leq 0.
$$
Thus $w$ satisfies the viscosity differential inequality at each point of $U$.
\end{proof}

We now state the parabolic Jensen-Ishii's maximum principle (\cite{CIL92}, \cite[p. 65]{DI04}):

\begin{theo} \label{Jen-Ish}
 Let $\Omega \subset \R^{N}$ be a domain, $u$  an upper semi-continuous function and $v$ a lower semi-continuous function in $]0, T[ \times \Omega$.
Let $\phi$ be a function defined in $]0, T[ \times \Omega^2$ such that $(t,x,y) \longmapsto \phi (t,x,y)$ is continuously differentiable in $t$ and twice continuously differentiable in $(x,y)$.  

Assume that the function $(t,x,y) \longmapsto u (t,x) - v(t,y) - \phi (t,x,y)$ has a local maximum at some point $(\hat t, \hat x, \hat y) \in ]0, T[ \times \Omega^2$. 

Assume furthermore that both $w=u$ and $w=-v$ satisfy:
\begin{displaymath}
(\label{Cond}\ref{Cond})\left\{
\begin{array}{ll}
\forall (s,z) \in \Omega & \exists r>0 \ \text{such that} \ \forall M >0 \  \exists C \ \text{satisfying} \\
&\left.
\begin{array}{l}
|(t,x)-(s,z)| \le r,\\
 (\tau,p,Q)\in \mathcal{P}^{2,+}w(t,x) \\
|w(t,x)|+|p| + |Q| \le M
\end{array} \right\} \Longrightarrow \tau\le C.
\end{array}\right.
\end{displaymath}

Then for any $\kappa > 0$, there exists 
$(\tau_1,p_1,Q^+) \in \bar{\mathcal P}^{2,+} u (\hat t, \hat x)$, 
$(\tau_2,p_2,Q^-) \in \bar{\mathcal P}^{2,-} v (\hat t, \hat y)$ such that
$$\tau_1 = \tau_2 + D_t \phi (\hat t, \hat x,\hat y), \ p_1 = D_x \phi (\hat t, \hat x,\hat y), \ p_2 = - D_y \phi (\hat t, \hat x,\hat y)$$ and
$$
-\left(\frac{1}{\kappa} + \| A \| \right) I \leq 
\left(
\begin{array}{cc}
 Q^+ &0 \\
0 & - Q^-
\end{array}
\right) \leq  A  + \kappa A^2,
$$
in the sense of quadratic forms on $\R^N$, where $A := D_{x,y}^2 \phi (\hat t, \hat x,\hat y)$. 
\end{theo}


\begin{rem}
Condition (\ref{Cond}) is automatically satisfied for $w$ locally  Lipschitz in the time variable
or if $w$ is a subsolution of (1.1) with $\mu>0$. It need not be satisfied for a general supersolution
of (1.1) even if $\mu>0$. 
\end{rem}

\subsection{Regularizing in time}

Given  a bounded upper semi-continuous function $u : [0,T[ \times \Omega \longrightarrow \R$,
we consider the upper approximating sequence by Lipschitz functions in $t$,  
$$
u^k (t,x) := \sup \{u (s,x) - k \vert s - t\vert , s \in [0,T[\}, \, \, (t,x) \in [0,T[ \times \Omega.
$$

If $v$ is a bounded lower semi-continuous function, we consider the lower approximating sequence of Lipschitz functions in $t$,
$$
v_k (t,x) := \inf \{v (s,x) + k \vert s - t\vert , s \in [0,T[\}, \, \, (t,x) \in [0,T[ \times \Omega.
$$

\begin{lem} \label{reg} 
For $k \in \R^+$, $u^k$ is an upper semi-continuous function which satisfies the following properties:
 \begin{itemize}
 \item $u (t,z) \leq u^k (t,z) \leq \sup_{\vert s - t\vert \leq A\slash k} u (s,z)$, where $A > 2 osc_{X_T} u$. 
 \item $\vert u^k (t,x) - u^k (s,x) \vert \leq k \vert s - t\vert$, 
for $(s,z) \in [0,T[ \times \Omega, (t,z) \in [0,T[ \times \Omega$. 
 \item For all $(t_0,z_0) \in [0,T-A/k] \times \Omega$, there exists $t_0^* \in [0,T[$ such that 
$$
\vert t_0^* - t_0\vert \leq A\slash k \text{ and } u^k (t_0,z_0) = u (t_0^*,z_0) - k \vert t_0 - t_0^*\vert.
$$
 \end{itemize}

 Moreover if $u$ satisfies  
 \begin{equation} \label{eq:subdiffineq}
 e^{\partial_t{u}  + F (t,u_t,\cdot)} \mu (t,\cdot) \leq (dd^ c u_t)^ n
\text{ in }  ]0,T[ \times \Omega,
 \end{equation} 
  where $\mu = \mu (\cdot,\cdot) \geq 0$ is a continuous Borel measure in $\Omega_T$,
 then  the function  $u^k$ is a subsolution of 
 $$
 e^{\partial_t{w}   + F_k (t,u_t,\cdot)} \mu_k (t,\cdot)  -  (dd^ c w_t)^ n = 0
 \text{ in } ]A/k,T-A/k[ \times \Omega,
 $$
 where $F_k (t,x,z) :=   \inf_{\vert s -t\vert \leq A \slash k} F (s,x,z) + k \vert s - t\vert $ and $\mu_k (t,z) := \inf_{\vert s -t\vert \leq A \slash k} \mu (s,z)$.
  The dual statement is true for a lower semi-continuous function $v$ which is a supersolution.
 \end{lem}

\begin{proof} 
The first statement is elementary. Let us prove the second one in the same spirit as \cite{CC95}. Let $(t_0,z_0) \in ]0,T[ \times \Omega$ be fixed  and let $q (t,z)$ be an upper test function that touches $u^k$ from above at  $(t_0,z_0)$.  Consider for $k$ large enough, the following smooth function given by
$$
q^* (t,z) := q (t + t_0 - t_0^*, z) + k \vert t_0 - t_0^*\vert.
$$
Then $q^*$ is an upper test function for $u$ at the point $(t_0^*,z_0)$. Since $u$ satisfies the differential inequality (\ref{eq:subdiffineq}), we have
$$
 e^{\partial_t q^* (t_0^*,z_0) + F (t_0^*, q^* (t_0,z_0), z_0)} \mu (t_0^*,z_0) \leq (dd^c q^*_{t_0^*} (z_0))^n.
$$
Since $\partial_t q^* (t_0^*,z_0) = \partial_t q (t_0,z_0)$, 
$q^* (t_0^*,z_0) = q (t_0,z_0)+k|t-t^*_0|$ and $dd^c q^*_{t_0^*} (z_0) = dd^c q_{t_0} (x_0)$ and  $F$ is non decreasing, we deduce the following inequality
$$
e^{\partial_t q (t_0,z_0) +  F (t_0^*, q (t_0,z_0), z_0)} \mu (t_0^*,z_0) \leq (dd^c q_{t_0} (z_0))^n,
$$
which proves the statement of the lemma since  $\mu (t_0^*,z_0) \geq \mu_k (t_0,z_0)$ and $ F (t_0^*, q (t_0,z_0), z_0) \geq  F_k (t_0, q (t_0,z_0), z_0)$.

For a supersolution the same proof works modulo obvious modifications. 
\end{proof}

\subsection{Spatial plurisubharmonicity of parabolic subsolutions}

We first connect sub/super-solutions of certain degenerate elliptic complex Monge-Amp\`ere equations
to sub/super-solutions properties of their slices.

\begin{prop} \label{lem:PartialSol}
 Let $ G:]0,T[ \times \R \times \Omega  \longrightarrow \R$ be a continuous function,  
$\nu (t,z) = \nu_t (z)$ a continuous family of volume forms   and let $w : ]0,T[ \times \Omega \longrightarrow \R$  be a subsolution (resp. supersolution) to the degenerate elliptic complex Monge-Amp\`ere equation 
 $$
 e^{ G (t,z,w)} \nu(t,z) - (dd^c w)^n = 0,
 $$ 
 in the viscosity sense in $]0,T[ \times \Omega$.  
Then for all $t_0 \in[0,T[$, the function 
$$
w_{t_0} : z \longmapsto w (t_0,z)
$$ 
is a subsolution (resp. supersolution) to the degenerate elliptic  equation  
$e^{ G (t_0,z,\p)} \nu_{t_0} - (dd^c w_{t_0})^n = 0$ in $\Omega$. 
\end{prop}

Let us stress that these equations do not contain any time derivative $\partial_t w$!

\begin{proof} 
We give the proof for the supersolution case and let the reader deal with the (slightly simpler)
case of subsolutions.

Aassume that $w$ satisfies the differential inequality 
$$
 (dd^c w)^n  \leq e^{ G (t,z,w)} \nu (t,z)
$$ 
in the sense of viscosity in $ U := ]0,T[ \times \Omega$. 
We approximate $w$ by inf-convolution $w_\e (t,z)$ in all variables, the function
$w_\e$ is defined in
the open set $U_\e := ]A \e, T - A \e[ \times \Omega_\e \subset U$ for $\e> 0$ small, where 
$$
\Omega_{\e}= \{z\in \Omega \ | \ d(z,\partial \Omega)> A \e\},
\text{ with }
A := 4 \ \text{osc}_X(u). 
$$

It is a classical fact (see \cite{CC95}) that the function $v := w_\e$ is a supersolution of an approximate parabolic Monge-Amp\` ere equation:
it satisfies the  differential inequality
$$ 
 (dd^c v_\e)^n \leq e^{ G^\e (t,z,v_\e)} \nu^\e,
$$
in the sense of viscosity, where 
$$
\nu_\e (t,z) := \sup \{ \nu (t',z') ;  \vert t' - t\vert , \vert z' - z\vert \leq A \e\}
$$ 
and $G^\e$ is defined similarly.

 Since $w_\e$ is semi-concave in $U_\e$, it follows from Alexandrov's theorem that it is  twice differentiable almost everywhere in $U_\e$. The above inequality is therefore satisfied pointwise
almost everywhere, i.e. at each point $(t,z)$ where $w_\e$ has a second order jet.  Observe that for almost all $t_0 \in ]A \e,T- A\e[$, there exist a set $E^{t_0} \subset \Omega_\e$ of Lebesgue measure $0$ such that for any $z_0 \notin E^{t_0}$, the function  $w_\e $ is twice differentiable at $(t_0,z_0)$.
 By definition we have $\mathcal J^2 w_\e (t_0,z_0) = \{(\tau,p,\kappa,Q)\}$ and  $\{(p,Q)\} = {\mathcal J}^2 \p (z_0)$, where $\p = w_\e (t_0,\cdot)$. 
 The viscosity inequality satisfied by $w$ at $(t_0,z_0)$ implies
$$ 
(dd^c Q)_+^n \leq  e^{ G^\e (t_0,z_0,v (t_0,z_0))} \nu^\e (t_0,z_0) .
$$

It follows that for almost all fixed $t_0 \in ]0,T[$, the function $\p (z) := w_\e (t_0,z)$
is pointwise second order differentiable at almost all $z_0 \in \Omega_\e$ and satisfies 
$$
(dd^c \p (z_0))_{+}^n \leq  e^{ G^\e (t_0,z_0, \p (z_0))} \nu^\e (t_0,z_0).
$$
 Lemma~\ref{lem:subpp} now shows that
$\p$ satisfies the viscosity inequality $(dd^c \p)^n \leq e^{ F^e (t_0,\cdot, \p)} \nu^\e (t_0,\cdot)$ 
at every point of  $\Omega_\e$.
Since $\nu^\e \to \nu$ and $G^\e \to G$ locally uniformly in $U$, it follows from the stability Lemma~\ref{lem:Stab} that $w (t_0,\cdot) = \lim_{\e \to 0} w_\e (t_0,\cdot)$ is a supersolution to the degenerate elliptic equation
$$
e^{ G (t_0,\cdot, \p)}  \nu_{t_0} - (dd^c \p)^n = 0
$$ 
in the  sense of viscosity in $\Omega$. 
This is true for almost every $t_0 \in ]0,T[$. Now given any $t_0 \in ]0,T[$, one can find a sequence of points $(t^j)$ converging to $t_0$ in $]0,T[$
such that for every $j \in \N$, the function $\p^j := \p (t^j,\cdot)$ is a supersolution to the degenerate elliptic equation associated to $(G^j,\nu^j)$, where $G^j := G (t^j,\cdot)$ and $\nu^j := \nu (t^j,\cdot)$.
Since $G^j \to G (t_0,\cdot)$ and $\nu^j \to \nu (t_0,\cdot)$ locally uniformly in $\Omega$, it follows from the stability Lemma in the degenerate elliptic case (see \cite{CIL92}) that $\p (t_0,\cdot)$ is a supersolution to the degenerate elliptic equation associated to $(G (t_0,\cdot),\nu (t_0,\cdot)$. 
\end{proof}

As a consequence we show that subsolutions to parabolic complex Monge-Amp\`ere equations
are plurisubharmonic in the space variable:

\begin{coro} \label{SPSH} 
 Assume that $u$ is a bounded subsolution to the parabolic Monge-Amp\`ere equation (\ref{eq:PMF})  in
$]0,T[ \times \Omega$. Then for any fixed  $t_0 \in [0,T[$, the function
$$
z \mapsto u(t_0,z) 
\text{ is plurisubharmonic} \, \text{ in } \, \Omega,
$$

Moreover  for all  $(t_0,z_0) \in \Omega_T$ and $(\tau,p,Q) \in  {\mathcal {P}}^{2,+} u (t_0,z_0)$, we have $ dd^c Q \geq 0$ and  $dd^c Q > 0$ when $\mu (t_0,z_0) > 0$. 
\end{coro}

\begin{proof} 
We consider here the parabolic Monge-Amp\`ere equation (\ref{eq:PMF}) as a degenerate elliptic equation on $]0,T[ \times \Omega$ as explained in Remark \ref{rem:parab=ellipt}.

Since $u$ is a subsolution to the parabolic Monge-Amp\`ere equation (\ref{eq:PMF}), 
it  is also a subsolution to the degenerate elliptic equation $ (dd^c u_t)^n = 0$ in $]0,T[ \times \Omega$.
 Applying Proposition~\ref{lem:PartialSol} with $\mu \equiv 0$, we conclude that for any fixed $t_0 \in]0,T[$, the function $ w := u (t_0,\cdot)$ is a subsolution of the degenerate elliptic equation $ (dd^c w)^n = 0$ in $\Omega$.

 Therefore by \cite{EGZ11} the function $\f = u (t_0,\cdot)$ is psh in $\Omega$. The last statement follows also from \cite{EGZ11}.
\end{proof}

\begin{prop} 
 Assume that $\mu \equiv 0$ vanishes identically in some open set $D \subset \Omega$ and $v$ is a bounded supersolution to the parabolic Monge-Amp\`ere equation (\ref{eq:PMF}) in
$]0,T,[ \times D$.  

Then for all  $t_0 \in ]0,T[$
the function $z \mapsto v(t_0,z) $ is a supersolution to the degenerate elliptic equation $(dd^c w)^n = 0$ in $D$ i.e.  $(dd^c v_{t_0})^n \leq 0$ in the viscosity sense in $D$. 

If  $v$ is moreover continuous in $]0,T,[ \times D$ then the  plurisubharmonic 
envelope $P (v_{t_0})=\sup \{ u \, | \, u \text{ psh in } D \text{ and }  u \leq v_{t_0} \}$ of the function 
$z \mapsto v(t_0,z) $  satisfies 
$$
(dd^c P (v_{t_0}))^n = 0
$$
in the viscosity sense in $D$, hence it is a maximal psh function in $D$.
\end{prop}

Recall that a psh function $u$ is {\it maximal} (see \cite{Sad81}) if for every relatively compact open set $U \subset D$
and every psh continuous function $h$ on $\overline{U}$, 
$$
h \leq u
\text{ on } \partial U \Rightarrow  
h \leq u \text{ in } U.
$$

\begin{proof} 
Since $v$ is a bounded supersolution to the parabolic Monge-Amp\`ere equation (\ref{eq:PMF}) in
$]0,T,[ \times D$ and $\mu \equiv 0$ in $D$, it follows that $v$ is a supersolution to the  degenerate elliptic equation  $(dd^c v)^n = 0$ in $]0,T[ \times D$.
Using Lemma~\ref{lem:PartialSol} with $\mu \equiv 0$, we infer  that for $t_0 \in]0,T[$, 
the function $ w := v (t_0,\cdot)$ is a supersolution of the degenerate elliptic equation 
$ (dd^c w)^n = 0$ in $D$.

When $w$ is continuous, it follows from \cite[Lemma 4.7]{EGZ11} that its plurisubharmonic envelope
$\theta := P (w)$ is a (viscosity) supersolution to the equation $ (dd^c \theta)^n = 0$ in $D$. 
Since $\theta$ is also plurisubharmonic,  we infer that $\theta$ is a viscosity solution to the homogeneous complex Monge-Amp\`ere equation $(dd^c \theta)^n = 0$ in $D$. 

Fix a ball $\B \Subset D$ and observe that the continuous psh function $\theta$
is the unique solution to the Dirichlet problem for the homogeneous complex Monge-Amp\`ere equation $(dd^c \p)^ n = 0$ in $\B$ with boundary values $\p _{|\partial \B} = \theta_{|\partial \B}$.

It is known \cite{EGZ11,W12} that the viscosity solution to this Dirichlet problem is 
 the upper envelope of all viscosity subsolutions.
Since viscosity subsolutions are exactly the pluripotential ones
 by \cite[Theorem 1.9]{EGZ11}, 
we infer that $\theta$ is the upper envelope of the pluripotential subsolutions to the Dirichlet problem above, hence it coincides with the Perron-Bremermann envelope and is a maximal psh function 
(see \cite{Bre59}, \cite{Sad81}). Thus $(dd^c \theta)^n = 0$ in the pluripotential sense 
and $\theta$ is a maximal psh function in the open set $D$.
 \end{proof}

\begin{rem}
Let $\f$ be a continuous plurisubharmonic function and $\mu \geq 0$ an
absolutely continuous measure with continuous non-negative density.
As the proof of the proposition above shows, the following are equivalent:

i) $(dd^c \f)^n=\mu$ in the pluripotential sense
of Bedford-Taylor \cite{BT82};

ii) $(dd^c \f)^n=\mu$ in the viscosity sense \cite{EGZ11}.

\noindent The dictionary between viscosity and plutipotential theory is quite subtle
when $\mu$ is allowed to vanish and
as far as supersolution are concerned. These notions however coincide for continuous solutions
of Dirichlet problems.
\end{rem}

The following immediate consequence of the previous proposition shows that one cannot 
run continuously a parabolic complex Monge-Amp\`ere flow from an arbitrary initial data,
if the measure $\mu$ is allowed to vanish:

\begin{coro} \label{SuperMax} 
   Assume that  $\f$ is a solution to the the parabolic Monge-Amp\`ere equation (\ref{eq:PMF}) in
$]0,T,[ \times \Omega$ which extends continuously to $[0,T[ \times \Omega$. 
If $\mu$ vanishes in some open set $D$, then
 for all $t \in [0,T[$, the function $\f_t$ is a maximal psh function in $D$. In particular 
$\f_0$ has to be a maximal plurisubharmonic function in $D$.
\end{coro}

\section{The Parabolic Comparison Principle}

In this section we establish a comparison principle for the following parabolic complex Monge-Amp\` ere equation in bounded domains of $\C^n$:
\begin{equation} \label{eq:CMAF}
 e^{\partial_t \f + F (t,\cdot, \f_t)} \mu_t  - (dd^c \f_t)^n = 0,
\end{equation}
where $\mu (t,z) = \mu_t (z) \geq 0$ is a continuous family of Borel measure on $\Omega$.
 
\medskip

We begin with a technical lemma. 

\begin{lem}\label{lem:MaxEst1bis} 
 Let   $\mu (t,z) \geq 0$ and $\nu (t,x) \geq 0$ be  two continuous Borel 
measures  on $\Omega$ depending on the variables $(t,z)$ and 
$F,G : [0,T[  \times \overline{\Omega} \times \R \longrightarrow \R$ 
two continuous functions. 

Assume that $u: \bar \Omega_T \to \R$ is upper semicontinuous
and  $v: \bar \Omega_T \to \R$ is lower semicontinuous. 
Assume that the restriction of $u$ to $\Omega_T$
is a bounded subsolution to the parabolic complex Monge-Amp\`ere equation 
(\ref{eq:CMAF}) associated to $(F,\mu)$ in $\Omega_T$ and   that the restriction of $v$ to $\Omega_T$ is a bounded  
supersolution to the parabolic complex Monge-Amp\`ere equation (\ref{eq:CMAF}) associated to $(G,\nu)$  in $\Omega_T$.

Assume $(\dagger)$ that $v$ is locally Lipschitz in the time variable and either $u$ is locally Lipschitz in the time variable
or $\mu>0$. 

Then, for every $\delta>0$ either $\sup_{\bar\Omega_T}(u(t,x)-v(t,x)-\frac{\delta}{T-t})$ is attained on $\partial_0 \Omega_T$
or there exists $(\hat{t},\hat{x})\in ]0,T']\times \Omega$ where 
\begin{equation} \label{eq:restrictiont} T'=T-2\delta(\|u \|_{\infty}+ \| v\|_{\infty})
\end{equation}
such that 
$\sup(u(t,x)-v(t,x)-\frac{\delta}{T-t})$ is attained at $(\hat{t},\hat{x})$ and

\begin{equation} \label{eq:MaxEst1bis} 
 e^{ \frac{\de}{(T-\hat t)^2 } +  F (\hat t,\hat x, u (\hat t,\hat x)) - G (\hat t,\hat x,v (\hat t,\hat x))} \mu (\hat t,\hat x) \leq \nu (\hat t,\hat x).
\end{equation}
\end{lem}

\begin{proof}
Consider 
$$
w (t,x) := u (t,x) - v(t,x) - \frac{\delta}{T - t}.
$$
This function is upper semi-continuous and bounded from above on the compact set $\overline{\Omega}_T$ and it is locally Lipschitz in the time variable. Since  $w (t,z)$ tends to $-\infty$ as 
$t \to T^-$,  there exists a point $(t_0,z_0) \in [0, T[ \times \bar{\Omega}$,  such that
$$
\overline{M} := \sup_{\overline{\Omega}_T} w = w (t_0,z_0).
$$

By construction this maximum cannot be achieved on $]T',T[\times \bar{\Omega}$. 
We can assume that this maximum   of $w$ on $\overline{\Omega}_T$  is not attained on 
$\partial_0 \Omega_T$.
The set $\{ (t,x)\in \bar \Omega_T, \ w(t,x)=\overline{M} \}$ is then a compact subset contained in $]0,T']\times \Omega$. 
Consider for small $\e > 0$, the function defined on $]0,T[ \times \Omega^2$ by 
$$
 w_\e (t,x,y) := u (t,x) - v(t,y) - \frac{\delta}{T - t} - \frac{1}{2 \e} \vert x - y\vert^2\cdot 
$$
This function is upper semi-continuous and bounded from above in $[0,T[ \times \overline{\Omega}^2$ 
by a uniform constant $C$, and it tends to $- \infty$ as $t \to T^-$, so it reaches its 
maximum on $[0,T[ \times \overline{\Omega}^2$ at some point $(t_\e,x_\e,y_\e) \in [0,T[ \times  \overline{\Omega}^2$ i.e.
$$
\overline{M}_\e := \sup_{t \in [0,T[ \times \overline{\Omega}^2} w_\e = u (t_\e,x_\e) - v (t_\e,y_\e) - \frac{\delta}{T - t_\e} - \frac{1}{2 \e} \vert x_\e - y_\e\vert^2\cdot
$$

Observe that $ \overline{M} \leq \overline{M}_\e \leq C$, which implies that any limit point 
of $(t_\e)$ is in $[0,T[$. It follows from \cite[Proposition 3.7]{CIL92} 
that $\vert x_\e - y_\e \vert^ 2 = o (\e)$ and that there is a subsequence $ (t_{\e_j},x_{\e_j},y_{\e_j})$
 converging to $(\hat t,\hat x,\hat x) \in [0,T[ \times \overline{\Omega}^2$ where
 $(\hat t,\hat x)$ is a maximum point of $w$ on $\overline{\Omega}_T$
and 
\begin{equation}
\label{eq:limmax}\lim_{j\to \infty} \overline{M}_{\e_j}=\overline{M}.
\end{equation}

To simplify notation we set  for any $j \in \N$,
  $(t_j,x_j,y_j) = (t_{\e_j},x_{\e_j},y_{\e_j})$. Extracting and relabelling we may assume 
that $(u(t_j,x_j,y_j))_j$ and $(v(t_j,x_j,y_j))_j$ converge. 
By the semicontinuity of $u$ and $v$,
\begin{equation} \label{eq:semicont}
\lim_{j\to \infty} u(t_j,x_j,y_j) \le u(\hat{t}, \hat{x}), \  \lim_{j\to \infty} v(t_j,x_j,y_j)\ge v(\hat{t}, \hat{x}).
\end{equation}

On the other hand $(\ref{eq:limmax})$ implies that:
\begin{eqnarray*} \lim_{j\to \infty} u(t_j,x_j,y_j) - \lim_{j\to \infty} v(t_j,x_j,y_j) -\frac{\delta}{T-\hat{t}}
&=&u(\hat{t}, \hat{x})-v(\hat{t}, \hat{x})   -\frac{\delta}{T-\hat{t}}, \\
\lim_{j\to \infty} u(t_j,x_j,y_j) - \lim_{j\to \infty} v(t_j,x_j,y_j) 
&=&u(\hat{t}, \hat{x})-v(\hat{t}, \hat{x}).
\end{eqnarray*}

Together with $(\ref{eq:semicont})$, this yields
\begin{equation}\label{eq:limit} \lim_{j\to \infty} u(t_j,x_j)=u(\hat{t},\hat{x}), \ \lim_{j\to \infty} v(t_j,y_j)=v(\hat{t},\hat{x}). 
\end{equation}

From our assumption we conclude that  $(\hat t,\hat x) \in ]0,T[ \times \Omega$ and
\begin{equation} \label{eq:Lineq2}
\overline{ M}  = u (\hat t,\hat x) - v(\hat t,\hat x) - \frac{\delta}{T - \hat t}\cdot
\end{equation}
  
  Applying the parabolic Jensen-Ishii's maximum principle Theorem~\ref{Jen-Ish} 
(the technical assumption (\ref{Cond}) being satisfied thanks to $(\dagger)$) to the functions 
  $U (t,x) := u (t,x) - \frac{\delta}{T - t}$,  
$ v$ and the penality function $\phi (t,x,y) :=  \frac{1}{2 \e} \vert x - y\vert^2$ 
for any fixed $\e = \e_j$, we  find  approximate parabolic second order jets
  $(\tau_j,p_j^{\pm},Q_j^{\pm}) \in \R \times \R^{2 n} \times \mathcal S_{2 n}$   such that 
 $$
 \left(\tau_j + \frac{\delta}{(T - t_j)^2} ,p_j^+,Q_j^{+}\right) \in \mathcal {\bar P}^{2,+} u (t_j,x_j), \ \ \left(\tau_j,p_j^-,Q_j^{-}\right) \in \mathcal {\bar P}^{2,-} v (t_j,y_j)$$
  with $ p_j^+ = - p_j^- =   \frac{1}{\e_j} (x_j - y_j)$ and $Q_j^+ \leq Q_j^-$ 
(see \cite[p.17]{CIL92} for the classical deduction of 
this inequality  from Theorem \ref{Jen-Ish} using the form 
 of the flat space penalization function and compare \cite{AFS08} for the difficulties in curved space). 


Applying the parabolic viscosity differential inequalities, we obtain for all $j \in \N$,
\begin{eqnarray*}
e^{\tau_j + \frac{\de}{(T-t_j)^2 } + F (t_j, x_j, u (t_j,x_j))} \mu (t_j,x_j) &\leq &  (dd^c Q_j^+)^n \\
&\leq &  (dd^c Q_j^-)^n  \\ 
& \leq & e^{\tau_j + G (t_j, y_j,v (t_j,y_j))} \nu (t_j,y_j),
\end{eqnarray*}
which implies that
$$
e^{ \frac{\de}{(T-t_j)^2 } + F (t_j, x_j,u (t_j,x_j)) - G (t_j,y_j,v (t_j,y_j))} \mu (t_j,x_j) \leq \nu (t_j,y_j),
$$
Letting $j \to + \infty$ and using (\ref{eq:limit}) (\ref{eq:Lineq2}) , we obtain the
 inequality (\ref{eq:MaxEst1bis}). 
\end{proof}

\begin{theo}\label{thm:comparison} 
 Assume that $\mu (z) \geq 0$ is a continuous  non negative volume form on $\overline{\Omega}$.
Let $u$  be a bounded subsolution to the parabolic complex Monge-Amp\`ere equation (\ref{eq:CMAF})  and $v$ a bounded  supersolution to the parabolic complex Monge-Amp\`ere equation (\ref{eq:CMAF}) in $\Omega_T$.  Then 
$$
\max_{\overline{\Omega}_T} (u  - v) \leq \max \{\max_{\partial_0 \Omega_T} (u  - v), 0\},
$$
where $u$ (resp. $v$) has been extended as an upper (resp. a lower) semicontinuous function 
to $\bar{\Omega}_T$. 
\end{theo}

\begin{proof}
 

 

 {\bf Step 1.} Assume that $\mu > 0$ in $\Omega_T$ and that $v$ is locally Lipschitz in the time variable.
Fix $\delta>0$. 
Apply Lemma \ref{lem:MaxEst1bis} with $\mu=\nu$, $F=G$. It follows that either 
$$
\sup_{\bar\Omega_T} \left\{u(t,x)-v(t,x)-\frac{\delta}{T-t}\right\} = \sup_{\partial_0\bar\Omega_T} \left\{u(t,x)-v(t,x)-\frac{\delta}{T-t}\right\}
$$
or 
$$ e^{ \frac{\de}{(T-\hat t)^2 } +  F (\hat t,\hat x, u (\hat t,\hat x)) - F (\hat t,\hat x,v (\hat t,\hat x))} \le 1
$$
which implies $F (\hat t,\hat x, u (\hat t,\hat x)) - F (\hat t,\hat x,v (\hat t,\hat x))<0$ hence $u (\hat t,\hat x)))<
v (\hat t,\hat x)$. 
In either case, every $(t,x)\in \bar \Omega_T$ satisfies
 $$
u(t,x)-v(t,x) -\frac{\delta}{T-t} < \max \{0, \sup_{\partial_0\bar\Omega_T}(u-v)\}.$$
Since $\delta$ can be choosen arbitrary small, we infer: 
$$
u-v\le \max \{0, \sup_{\partial_0\bar\Omega_T}(u-v)\}.
$$

\smallskip

{\bf Step 2.} Still assuming that $\mu > 0$ in $\Omega_T$, let us remove the assumption
 that $v$ is locally Lipschitz in the time variable. Fix $\delta>0$. Either
$$\sup_{\bar\Omega_T}  \left\{u(t,x)-v(t,x)-\frac{\delta}{T-t}\right\} = \sup_{\partial_0\bar\Omega_T} \left\{u(t,x)-v(t,x)-\frac{\delta}{T-t} \right\}$$
or
$$
\sup_{\bar\Omega_T}\left\{u(t,x)-v(t,x)-\frac{\delta}{T-t}\right\} > \sup_{\partial_0\bar\Omega_T} \left\{u(t,x)-v(t,x)-\frac{\delta}{T-t}\right\}
$$

Suppose we are in the second case. 
Fix $\bar{s}\in \R$ such that
$$
\sup_{\bar\Omega_T} \left\{u(t,x)-v(t,x)-\frac{\delta}{T-t}\right\} > \bar{s} >
\sup_{\partial_0\bar\Omega_T} \left\{u(t,x)-v(t,x)-\frac{\delta}{T-t})\right\}.
$$

Since $ \bar{s} >\sup_{\partial_0\bar\Omega_T}(u(t,x)-v(t,x)-\frac{\delta}{T-t})$, we have
$\partial_0\bar\Omega_T \subset \{ w(t,x)<\bar{s} \}$. Since $\{ w(t,x)<\bar{s} \}$ is open
and contains $\{ 0 \} \times \bar\Omega$,
 we can find $\eta>0$ such $[0, \eta]\times \bar\Omega \subset \{ w(t,x)<\bar{s} \}$ so that every
$(t,x)\in [0,\eta]\times \bar\Omega \cup \partial_0\bar\Omega_T $ 
satisfies 
$$
u(t,x)-v(t,x)-\frac{\delta}{T-t}< \bar{s}.
$$ 
We now apply Lemma \ref{reg}
to $v$. Then, by Dini-Cartan's lemma we have
$$ \lim_{k\to \infty} \sup_{\bar\Omega_T} \{u(t,x)-v_k(t,x)-\frac{\delta}{T-t})=\sup_{\bar\Omega_T}(u(t,x)-v(t,x)-\frac{\delta}{T-t}\}
$$
and similarly
$$
 \lim_{k\to \infty}
 \sup_{[0,\eta]\times \bar\Omega \cup \partial_0\bar\Omega_T} \{u-v_k-\frac{\delta}{T-t}\}=
\sup_{[0,\eta]\times \bar\Omega \cup \partial_0\bar\Omega_T}\{u-v-\frac{\delta}{T-t}\}.
$$

Hence we can assume that for $k$ large enough the maximum of $w_k (t,x) := u(t,x)-v_k(t,x)-\frac{\delta}{T-t}$ 
is not attained on $[0,\eta]\times \bar{\Omega}\cup \partial_0\Omega_T$. 
Choose $k$ large enough so that the supersolution property of $v_k$ is  valid for $\eta/2 < t < T'$. 
Lemma \ref{lem:MaxEst1bis} applied to $\tilde u(t,x)= u( t+\eta, x), \  \tilde v(t,x)= v_k (t+\eta, x)$, yields 
\begin{equation} \label{eq:MaxEst2} 
 F (\hat t,\hat x , u(\hat t,\hat x)) - F^k (\hat t ,\hat  x ,v_k (\hat t,\hat x)) + \frac{\de}{T^2 } \le 
 \log (\mu^k \slash \mu) (\hat t,\hat x)),
 \end{equation}
 where $(\hat t,\hat x) = (\hat t_k,\hat x_k) \in ]0,T[ \times \Omega$ is a 
point where the function $w_k (t,x)$ takes its maximum in $\overline{\Omega}_T$. 
 
 Since $F$ and $\mu$ are uniformly continuous in $[0,T'] \times \overline{\Omega}\times [-K,K]$ with 
$K=\max(\|u\|_{\infty},\|v\|_{\infty})$, it follows that, for $k$ large enough, we have 
 \begin{equation} \label{eq:MaxEst3} 
 F (\hat t,\hat x , u(\hat t,\hat x)) - F (\hat t,\hat x, v_k (\hat t,\hat x)) \leq - \frac{\de}{ 2 T^2 }.
 \end{equation}

From this we get $u(\hat t,\hat x)<v_k (\hat t,\hat x)$. 
Hence $\sup_{\bar{\Omega}_T} w_k(t,x) <0$ and $\sup_{\bar\Omega_T}(u(t,x)-v(t,x)-\frac{\delta}{T-t})\le 0$. 

In particular, whether we are in the first or the second case, we infer
\begin{eqnarray*}
\sup_{\bar\Omega_T}(u(t,x)-v(t,x)-\frac{\delta}{T-t})&\le&
 \max(0, \sup_{\partial_0\bar\Omega_T}(u(t,x)-v(t,x)-\frac{\delta}{T-t})\\
&\le & \max(0, \sup_{\partial_0\bar\Omega_T}(u(t,x)-v(t,x)),
\end{eqnarray*}
and for every $(t,x)\in \Omega_T$ we have 
$$
u(t,x)-v(t,x)-\frac{\delta}{T-t} \le \max(0, \sup_{\partial_0\bar\Omega_T}(u(t,x)-v(t,x)).
$$ 
Since $\delta$
can be chosen arbitrary small, we conclude once again that 
$$
u-v \le \max \{0, \sup_{\partial_0\bar\Omega_T}(u-v)\}.
$$

\smallskip

{\bf Step 3.} Assume that $\mu (t,x) = \nu (t,x)\geq 0$ and the subsolution $ u $ is locally uniformly Lipschitz in $t$.

More precisely we assume that $t \longmapsto u (t,z)$ is $C$-Lipschitz in $t$  in some subset $[0,T'] \subset [0,T[$ uniformly in $z \in \Omega$. The idea is to perturb $\mu$ by adding an arbitrary small positive term.

Consider for $\eta > 0$ small enough, the positive volume form $\tilde \mu := \mu + {\eta} \beta^n$, where $\beta = dd^c \rho  > 0$ is the standard K\"ahler form on $\Omega$ i.e. $\rho (z) := \vert z \vert^2 - R^2,$ where $R > 1$ is large enough so that $\rho < 0$ in $\Omega$. 
Then fix $\e > 0$ and consider the function  $ \p (t,z) := u (t,z) + \e \rho (z)$. 
This is an upper semi-continuous function in $\Om_T$.
We claim that  $\p$ is a subsolution to the equation
\begin{equation} \label{eq:MaxEst5}
 e^{\partial_t \p (t,\cdot) + F (t,\p_t,\cdot)} \tilde \mu (t,\cdot) \leq (dd^c \p_t)^n,
\end{equation} 
in $[0,T'] \times \Omega$ for an appropriate choice of $\eta $ in terms of $\e$. 

Indeed since $\rho$ is $C^2$, any parabolic upper test function $\theta$ for $\p$ at any point $(t_0,z_0)$ can be written as $\theta (t,z) := \tilde \theta (t,z) + \e \rho (z)$, where $\tilde \theta$ is a parabolic upper test function for $u$ at the point $(t_0,z_0)$. 
 From the  viscosity inequality for $u$ we know that $dd^c \tilde \theta_{t_0} \geq 0$ and  
$$
 (dd^c \tilde \theta_{t_0})_{z_0}^n \geq e^{\partial_t \tilde{\theta} (t_0,z_0) + F (t_0,z_0,\tilde{\theta} (t_0,z_0))} \mu (t_0,z_0).
$$
Therefore $dd^c \theta_{t_0} (z_0) = dd^c \tilde \theta_{t_0} (z_0)+ \e \beta \geq 0$ and then 
$$
(dd^c \theta_{t_0})^n_{z_0} \geq  (dd^c \tilde \theta_{t_0})^n  + \e^n \beta^n \geq e^{\partial_t \theta (t_0,z_0) +  F (t_0,{\theta} (t_0,z_0),z_0)} \mu (t_0,z_0) +\e^n \beta^n,
$$
since $\theta \leq \tilde \theta$ and $F$ in non decreasing in the second variable.

Now set $M := \sup_{\Omega_{T'}} u$ and $A := \sup \{F (t,z,M) ; 0 \leq t \leq  T', z \in \Omega\}$.  Since $ u$ is $ C$-Lipschitz in $t$ uniformly in $z$ and $u \leq_{(t_0,z_0)} \theta$, it follows from Taylor's formula that $\partial_t \theta (t_0,z_0) \leq  C$. Then 
$$
 e^{\partial_t \theta (t_0,z_0) + F (t_0,z_0,\theta (t_0,z_0))}  \leq  e^{C + A} .
 $$ 
 Therefore if we choose $\eta := \e^n e^{-A - C}$,  we obtain the inequality
$$
(dd^c \theta_{t_0})^n_{z_0} \geq  e^{\partial_t \theta (t_0,z_0) + F (t_0,z_0,\theta (t_0,z_0))} (\mu (t_0,z_0) + \eta \beta^n),
$$
which proves our claim.

Since $\tilde \mu \geq \mu$, the function $v$ is also a supersolution to the 
parabolic equation associated to $(F,\tilde \mu)$. We can apply the 
comparison principle of the first part and conclude that 
$u (t,z) + \e \rho (z) - v(t,z) \leq \max_{\partial_0 \Omega_T} (u - v)^+ + O (\e).$ 
Letting $\e \to 0$ we obtain the conclusion of the theorem.

\smallskip

{\bf Step 4.} Finally assume that $\mu (z)= \nu (z)\geq 0$ does not depend on $t$.  
Regularizing $u$ in the time variable only according to Lemma~\ref{reg}, we obtain a decreasing sequence $u^k$ of $k$-Lipschitz functions in $t$ converging to $u$. We know by   (\ref{eq:subdiffineq}) that $U = u^k$ is a subsolution to the parabolic Monge-Amp\`ere equation associated to $(F_k,\mu)$  i.e. 
$$
  e^{\partial_t U (t,\cdot) + F_k (t,\cdot,U_t)} \mu \leq (dd^c U_t)^n,
$$ 
where $F_k (t,z,r) := \inf_{\vert s - t\vert \leq 1 \slash k} F (s,z,r)$
on $[A/k, T-A/k]\times \Omega$. Observe that $\mu$ does not change after regularization in the time variable since it does not depend on $t$.

Using the perturbation argument of Step 3, we see that the function 
$ \psi^k := u^k (t,z) + \e \rho (z)$ satisfies the differential inequality
$$
  e^{\partial_t \psi (t,\cdot) + F_k (t,\cdot,\psi_t)} (\mu + \eta \beta^n) \leq (dd^c \p_t)^n,
$$ 
where $\eta := \e^n e^{-A - k}$  on $[A/k, T-A/k]\times \Omega$.

We now regularise $v$ in the time variable only according to Lemma~\ref{reg} and argue as in Step 2 to conclude
that, in the second case,
\begin{equation} \label{eq:MaxEst4}  
 F_k (\hat t,\hat x , u_k(\hat t,\hat x) + \e \rho (\hat t,\hat x)) -
 F^k (\hat t ,\hat  x ,v^k (\hat t,\hat x)) + \frac{\de}{T^2 } \le 
 \log (\mu \slash \tilde{\mu}) (\hat t,\hat x)),
 \end{equation}
 where $(\hat t,\hat x) = (\hat t_k,\hat x_k) \in ]0,T[ \times \Omega$ is a 
point where the function $u_k( t, x) + \e \rho ( t,x)-v^k ( t, x))-\frac{\delta}{T-t}$
achieves its maximum. 
 Since $F$  is uniformly continuous in $[0,T'] \times 
\overline{\Omega}\times [-K-2\epsilon R^2,K]$ and $\mu\le \tilde {\mu}$ with 
$K=\max(\|u\|_{\infty},\|v\|_{\infty})$, it follows that for $k$ large enough, 
 \begin{equation} 
 F (\hat t,\hat x , u_k(\hat t,\hat x) + \e \rho (\hat t,\hat x))) - F (\hat t,\hat x, v^k (\hat t,\hat x))
 \leq - \frac{\de}{ 2 T^2 }.
 \end{equation}

This yields:
$$ u_k(\hat t,\hat x) + \e \rho (\hat t,\hat x) < v^k (\hat t,\hat x)
$$
and for all $(t,x)\in \bar\Omega_T$
$$ u_k( t, x) + \e \rho ( t, x) - v^k (\hat t,\hat x)-\frac{\delta}{T-t}<0. 
$$

We can then let $\e$ decrease to $0$, then let $k$ go to $+\infty$
arguing as in the last part of Step 2, to conclude that 
\begin{eqnarray*}
\sup_{\bar\Omega_T}(u(t,x)-v(t,x)-\frac{\delta}{T-t})&\le&
 \max(0, \sup_{\partial_0\bar\Omega_T}(u(t,x)-v(t,x)-\frac{\delta}{T-t})\\
&\le & \max(0, \sup_{\partial_0\bar\Omega_T}(u(t,x)-v(t,x)).
\end{eqnarray*}
Letting $\delta$ decrease to $0$, we conclude the proof. 
\end{proof}

\begin{rem}
As the proof shows, the comparison principle is valid under more general conditions than 
those stated in the theorem, in particular:
 when the volume form $\mu (t,z) > 0$ depends on $(t,z)$ and does not vanish on $\overline{\Omega_T}$. 
 \end{rem}

\begin{rem}
An important case for applications is when $F$ is strongly increasing in the last variable, meaning that there exists $\al > 0$ such that for any $(t,z)$, the function $r \longmapsto F (t,z,r) - \al r$ is non decreasing in $\R$. 
Then we can prove a more precise comparison principle. Namely assume that 
$\mu (t,z) > 0$ and $\nu (t,z) \geq 0$ are two continuous volume forms  in $\Omega_T$, 
 $u$ is a subsolution to the parabolic complex Monge-Amp\`ere equation 
associated to $(F,\mu)$ and  $v$ is a supersolution to the parabolic Monge-Am\`pere equation associated to $(F,\nu)$, then 
 $$
 \max_{\Omega_T} (u  - v) \leq  \max \{M_0, (1 \slash \al) \log \gamma\}
 $$
 where $M_0 := \max_{\partial_0 \Omega_T} (u  - v)^+$ and $\gamma := \max_{\Omega_T} \nu \slash \mu$.
 This follows from the fundamental inequality (\ref{eq:MaxEst1bis}).\\
 \end{rem}
 
 \begin{rem} \label{GCP} 
 A  change of variables in time leads to the more general twisted parabolic complex Monge-Amp\`ere equation
 \begin{equation} \label{eq:GPCMAF}
 e^{ h (t) \partial_t \f + F (t,z,\f)} \mu (t,z) - (dd^ c \f_t)^ n = 0,
 \end{equation}
 where $h > 0$ is positive continuous function in $[0,T[$.
 
 The comparison principle holds for the more general parabolic complex Monge-Amp\`ere equation (\ref{eq:GPCMAF}).
 This follows from the change of variables
 $$
  u (s,z) := \f (t,z), \, \, \text{whith}, \, \,  t = \gamma (s)
 $$
 where $\gamma $ is a positive increasing function in $[0,S[$ with values in $[0,T[$ such that 
$\gamma (0) = 0$.
 
 Indeed observe that $\partial_s u = \gamma' (s) \partial_t \f (t,z)$ for $(s,z) \in [0,S[ \times \Omega$. 
Thus if we set
 $\gamma' (s) =  h (t)$, then  $\f$ is a solution to the twisted parabolic complex Monge-Amp\`ere equation (\ref{eq:GPCMAF}) if and only if $u$ is a solution to the  parabolic complex Monge-Amp\`ere equation
 $$
 e^{\partial_s u +  G (s,z,u)} \nu (s,z) - (dd^ c u_s)^ n = 0,
 $$
where $ G (s,z,r) := F (\gamma (s),z, r)$ and $\nu (s,z) := \mu (\gamma (s),z)$. 

Since $r \longmapsto G (t,z,r)$ in non decreasing, we can apply the comparison principle proved above and 
obtain the claim.
Observe that the equation $\gamma' (s) =  h (t)$ means that 
the inverse function $g (t)= \gamma^{- 1} (t)= s$ satisfies
$g' (t) = 1 \slash h (t)$ and $g (0) = 0$, thus $\gamma$ is uniquely determined by $h$.
\end{rem}

\section{Existence of solutions}

We now study the Cauchy-Dirichlet problem for the parabolic complex Monge-Amp\` ere equation
  \begin{equation} \label{eq:PMAF}
 e^{\partial_t \f +  F (t,\cdot,\f)} \mu - (dd^c \f_t)^n = 0, \, \, \text{in} \, \, \Omega_T,
  \end{equation}
 with  Cauchy-Dirichlet conditions,
  \begin{equation} \label{eq:CD-PMAF}
  \f (t,z) = \f_0 (z), \ \, \,  (t,\zeta) \in \partial_0 \Omega,
  \end{equation}
  where $\mu (z)\geq 0$  is a continuous volume form in $\Omega$, $\f_0 : \Omega \longrightarrow \R$ is continuous in $\bar \Omega$ and plurisubharmonic in $\Omega$ and $F : [0,T[ \times \Omega \times \R \longmapsto \R$ is a continuous function non decreasing in the last variable.

We assume that $\Omega$ is a strictly pseudoconvex domain and let $\rho$ 
be a defining function for $\Omega$ which is strictly plurisubharmonic in a neighborhood of $\overline{\Omega}$, with $- 1 \leq \rho < 0$ in $\Omega$.

\subsection{Existence of sub/super-solutions}

We first introduce  the notions of sub/super-solution for the Cauchy-Dirichlet problem:

 \begin{defi} Let  $\f_0$ be a  Cauchy-Dirichlet data function for the parabolic Monge-Amp\`ere equation (\ref{eq:CD-PMAF}).  

 1. We say that an upper semi-continuous function $u : [0,T[ \times \overline{\Omega} \longrightarrow \R$ is a subsolution to the Cauchy-Dirichlet problem (\ref{eq:CD-PMAF}) if $u$ is a subsolution to the parabolic  equation (\ref{eq:PMAF}) in $\Omega_T$ which satisfies  $u \leq \f_0$ on the parabolic boundary 
$\partial_0 \Omega_T$.

  2. We say that a lower semi-continuous function $v : [0,T[ \times \overline{\Omega} \longrightarrow \R$ is a supersolution  to the Cauchy-Dirichlet problem (\ref{eq:CD-PMAF}) if $v$ is a supersolution to the parabolic  equation (\ref{eq:PMAF}) in $\Omega_T$ which satisfies  $v  \geq \f_0$  on the parabolic boundary 
$\partial_0 \Omega_T$.
 \end{defi}

Observe  that sub/supersolutions to the parabolic complex Monge-Amp\`ere equation
we are interested in always exist:

 \begin{lem} \label{sub-super} 
\text{ }

1.  The constant function $\overline v \equiv  \sup_{\overline{\Omega}} \f_0 $ is a supersolution to  the Cauchy-Dirichlet problem (\ref{eq:CD-PMAF}) with Cauchy-Dirichlet data $\f_0$.

2. The Cauchy Dirichlet problem for the parabolic equation (\ref{eq:PMAF}) with initial data $\f_0$ admits a subsolution $\underline u$ in $]0,T[ \times \Omega$, which is continuous in $[0,T[ \times \overline{\Omega}$ and  satisfies $\underline u \leq \overline v$ in $[0,T[ \times \overline{\Omega}$. 
\end{lem}

\begin{proof} 
The first statement is obvious. 
To prove the second one, we consider the function defined for $t \in [0,T[$ by
 \begin{equation} \label{eq:Comp}
 B (t) := \int_0^t b_+ (s) d s, \, \, \, \text{where} \, \, 
  b (t) := \sup \{ F (t,z,\f_0 (z)) ;  z \in \Omega \}. 
  \end{equation}
 Choose $A > 0$ large enough so that $A (dd^c \rho)^n \geq \mu$  in $\Omega$. 
The  function
 $$
(t,z) \mapsto  \underline u (t,z) := A \rho (z) + \f_0 (z) - B (t)
 $$
 is a subsolution to Cauchy-Dirichlet problem for the parabolic complex Monge-Amp\`ere equation (\ref{eq:CD-PMAF}) which clearly satisfies $\underline u \leq \overline v$.
\end{proof}

 \begin{rem}
 Observe that the supersolution given above is bounded, while the subsolution $\underline u$ is continuous in $[0,T[ \times \overline{\Omega}$, hence  locally bounded. 
 When  $ (t,z) \longmapsto F (t,z,\f_0 (z))$ is bounded from above 
on $[0,T[ \times \overline{\Omega}$,  there exists a globally bounded subsolution. 
 Indeed set  
 $$
 B := \sup \{ F (t,z,\f_0 (z)) ; 0 \leq t < T, z \in \Omega \},
 $$ 
 and let $A > 1$ be so large that  $A^n (dd^c \rho)^n \geq e^{B} \mu$. 
Then 
$$
\underline u (t,z) := A \rho (z) + \f_0 (z),
$$ 
does the job.
\end{rem}

Consider the  upper envelope
 \begin{equation} \label{eq:env}
 \f := \sup \{ u \ ; u \in \mathcal S, \underline u \leq u \leq \overline v \},
 \end{equation}
 where $\mathcal S$ is the family of all subsolutions to the Cauchy-Dirichlet problem for the parabolic equation (\ref{eq:PMAF}) with the Cauchy-Dirichlet condition (\ref{eq:CD-PMAF}),
and $\underline u , \overline v$ are the sub/super-solutions from Lemma \ref{sub-super}.

 \begin{lem} \label{lem:Perron1}
 Given any non empty family $\mathcal{S}_0$ of bounded subsolutions to the parabolic equation (\ref{eq:PMAF})  which is bounded above by a continuous function, the usc regularization of the upper envelope $\phi_{ \mathcal{S}_0} :=\sup_{\phi \in \mathcal{S}_0} \phi$
is a subsolution to (\ref{eq:PMAF}) in $\Omega_T$.

If $\mathcal{S}$ is the family of all subsolutions to the Cauchy-Dirichlet problem  (\ref{eq:CD-PMAF}), its envelope
$ \phi_{ \mathcal{S}}$ coincides with the upper envelope $\f$ given by the formula  (\ref{eq:env}) and is a discontinuous viscosity solution to  
(\ref{eq:PMAF}) in $\Omega_T$.
 
 Moreover for any  $(t,z) \in [0,T[ \times \Omega$,
 \begin{equation} \label{eq:Perron1}
 \f^* (t,z)  - \f_* (t,z) \leq \sup_{\partial_0 \Omega_T} (\f^*   - \f_* )_+.
 \end{equation}
 \end{lem}

 \begin{proof} 
The first statement follows from the standard method of Perron (see \cite{CIL92}, \cite{IS12}). 
 Observe that the family $\mathcal{S}$  of all subsolutions to the Cauchy-Dirichlet 
problem  (\ref{eq:CD-PMAF}) is not empty since $\underline u \in \mathcal{S}$ and  bounded
from above  by $\overline v$, thanks to Lemma~\ref{sub-super} and  Theorem \ref{thm:comparison}. 

The fact that $\f$ is a discontinuous viscosity solution to  
(\ref{eq:PMAF}) in $\Omega_T$ follows by the general argument of Perron as in the degenerate elliptic case (see \cite{CIL92}, \cite{EGZ11}).
 \end{proof}

\subsection{Barriers}
 
  In order to prove that the above (a priori discontinuous) viscosity solution is a continuous viscosity
 solution and satisfies the Cauchy-Dirichlet condition, we need to construct appropriate barriers.

 \begin{defi}  
Fix $(t_0,x_0) \in \partial_0 \Omega_T$  and  $\e > 0$.

 1.  We say that an upper semi-continuous function $u : \Omega_T \longrightarrow \R$ is an 
$\e$-subbarrier  for the Cauchy problem (\ref{eq:CD-PMAF}) at the point $(t_0,x_0)$, 
if $u$ is a subsolution  to the  parabolic complex Monge-Amp\`ere equation (\ref{eq:PMAF})
 such that $u \le \f_0$ in $\partial\Omega_T$ and  $u_* (t_0,x_0) \geq \f_0 (x_0) - \e$. 

2. We say that a lower semi-continuous function $v :  \Omega_T \longrightarrow \R$ is an 
 $\e-$superbarrier to the Cauchy problem (\ref{eq:CD-PMAF}) at the boundary point $(t_0,x_0)$ if 
$v$ is a supersolution  to the parabolic equation (\ref{eq:PMAF}) such that $v \ge \f_0$ in $\partial_0\Omega_T$ and $v^*(t_0,x_0) \leq \f (x_0) + \e$.
 \end{defi}

\begin{defi}  \label{def:adm} 
We say $(\f_0,\mu)$ is admissible whenever for all $\epsilon>0$
we can find 
$\p_0 \in C^0(\bar \Omega)\cap PSH(\Omega)$ such that $\f_0 \le \p_0 \le \f_0+ \epsilon $ and 
$C=C_{\epsilon} \in \R$  such that $(dd^c \p_0)^n \le e^C \mu$ in the viscosity sense.
\end{defi}

In other words a Cauchy data $\f_0$ is admissible with respect to $\mu$ if it is the uniform limit on $\Omega$ of continuous psh functions whose Monge-Amp\`ere measure is controlled by $\mu$. 
In particular if $(dd^c \f_0)^n \le e^C \mu$ in the viscosity sense then
$(\f_0,\mu)$ is admissible. We also note the following useful criterion:

\begin{lem}  \label{lem:admissible}
 If $\mu >0$  then $(\f_0,\mu)$ is admissible.
\end{lem}

\begin{proof}  
This follows from classical results on approximation of plurisubharmonic functions. Indeed, any psh function in $\Omega$, continuous up to the boundary can be approximated uniformly in $\overline{\Omega}$ by psh functions in $\Omega$ that are smooth up to the boundary (see \cite{Sib87}, \cite{FW89}, \cite{AHP12}). 

Therefore given $\e > 0$, we can find a  function $\p_0$ psh in $\Omega$, smooth up to the boundary such that $\f_0 \le \p_0 \le \f_0+ \epsilon $ in $\Omega$.
If $\mu>0$ on $\{ 0 \}\times\bar{\Omega}$,  there is a constant $C > 0$ 
such that $(dd^c \p_0)^n \leq e^C \mu$ pointwise in $ \Omega$ , hence in the sense 
of viscosity in $\Omega$ \cite{EGZ11}. 
\end{proof}

\begin{exa}
If $\mu (z)\equiv 0$ vanishes identically on some open set $D \subset \Omega$ where $\f_0$ is not a maximal psh function
(i.e. where the Monge-Amp\`ere measure $(dd^c \f_0)^n$ is not zero) then $(\f_0,\mu)$ is not admissible.

Indeed $(\f_0,\mu)$ is admissible if and only if $\f_0$ is the uniform limit (in $\Omega$ hence in particular on $D$) of a sequence of continuous psh functions $\p_j$ such that 
$$
(dd^c \p_j)^n \leq C_j \mu
$$
for some $C_j>0$. In particular $\p_j$ has to be maximal in $D$, hence so is $\f_0$.
\end{exa}

\begin{prop} \label{prop:bar} Assume that $(\f_0,\mu)$ is admissible.
For all $\e > 0$ and  $(t_0,x_0) \in \partial_0 \Omega_{T}$, there exists a continuous function $U$ 
(resp. V) in $[0,T[ \times \overline{\Omega}$, which is an $\e-$subbarrier (resp. $\e$-superbarrier)
 to the Cauchy-Dirichlet problem (\ref{eq:CD-PMAF}) at $(t_0,x_0)$.
\end{prop}

\begin{proof} 
Fix $\e > 0$ and  $(t_0,z_0) \in  \partial_0 \Omega_{T}$. 

{\bf 1.} We first construct $\e$-subbarriers. There are two cases: 

1.1. Assume $t_0 = 0$ and $z_0 \in \overline{\Omega}$. Fix $\e > 0$ and define the following function
$$
 U (t,z) :=  \f_0 (z) + \e \rho (z) - B (t) - M t, \, \, (t,z) \in [0,T[ \times \Omega,
$$ 
where  $B (t)$ is the $C^1$ positive function defined by the formula (\ref{eq:Comp}) and $M > 0$ is a large constant to be chosen later. Recall that $B' (t) \geq F (t,z,\f_0 (z))$ in $[0,T[ \times \Omega$.

The function $U$ is  continuous in $[0,T[ \times \overline{\Omega}$,  
it is plurisubharmonic in the space variable $z \in \Omega$ and
$C^1$ in the time variable $t \in [0,T[$. Moreover it satisfies the inequality $(dd^c U_t)^n \geq \e^n (dd^c \rho)^n  $ in the pluripotential sense in $\Omega$ for any fixed $t  \in [0,T[ $.
Observe that  
$$
\partial_t U (t,z) +  F (t,z,U (t,z)) \leq  F (t,z,\f_0 (z)) -  C' (t) - M \leq - M,
$$
pointwise in $\Omega_{T}$.

 If we choose $M = M (\e) > 1$ large enough so that $\e^n (dd^c \rho)^n \geq e^{- M} \mu$, then  $U$ satisfies the inequality 
$$
(dd^c U_t)^n \geq e^{\partial_t U (t,\cdot) + F (t,\cdot,U_t)} \mu, 
$$
in the pluripotential sense in $\Omega$, for each $t$. Moreover it follows from \cite{EGZ11}  that  the function $U$ satisfies the differential inequality 
$$(dd^c U_t)^n \geq e^{\partial_t U (t,\cdot)  + F (t,\cdot,U_t)}$$ in the viscosity sense in $\Omega_T$. 
 Therefore the function $U$ is a viscosity subsolution to the Cauchy-Dirichlet problem (\ref{eq:CD-PMAF}).
  
 Since $U (0,\cdot) = \f_0 + \e \rho  \leq \f_0$ in $\Omega$,  $U (0,z_0) =  \f_0 (z_0) + \e  \rho (z_0)$ and $\rho \geq - 1$, we see that $U (0,z_0)  \geq \f_0 (z_0) - \e$. Hence $U$ is an $\e-$subbarrier at any point $(0,z_0) \in \{0\} \times \overline{\Omega}$. 

\smallskip

1.2. If $t_0 >0$ and $x_0 \in \partial \Omega$, we argue as in the proof of Lemma~\ref{sub-super}. 
We consider, for $t \in [0,T[$, 
$$
B (t) := \int_{t_0}^t b_+ (s) d s, \, \, \text{where} \, \, b (t) := \sup \{ F (t,z,\f_0 (z)) ; z \in \overline{\Omega} \}.
$$
This is $C^1$ function in $[0,T[$ satisfying $B (t_0) = 0$ and $B' (t) \geq F (t,z,\f_0 (z))$ for any  $(t,z) \in [0,T[ \times \overline{\Omega}$.
Choosing $A > 1$ large enough so that $ \mu \leq A^n (dd^c \rho)^n$, the function
$$
(t,x) \in [0,T[ \times \Omega \mapsto U (t,x) :=  \f_0 (x) + A \rho  (x) - B (t) \in \R
$$ 
is  a subbarrier at any  point $(t_0,x_0) \in [0,T[ \times\partial \Omega$.

\smallskip

 We haven't used the admissibility  of the Cauchy-Dirichlet data to construct subbarriers. 
 
{\bf 2.} Constructing superbarriers is a more delicate task that requires besides the admissibily some pluripotential tools.
 We also consider two cases:

 2.1. Fix $\e >0$ and use that $(\f_0,\mu)$ is admissible
to obtain a  psh function $\p_0$ in $\Omega$ continuous up to the boundary such that 
$\f_0 - \e\leq \p_0 \leq \f_0 $ in $\overline{\Omega}$. The maximal psh function 
$ \bar\psi_0$ solving the Dirichlet problem  
$$
(dd^c\bar\psi_0)^n=0
\; \text{ and } \; 
\bar\psi_0|_{\partial \Omega}=\p_0 |_{\partial \Omega}
$$
 is continuous and plurisubharmonic \cite{BT76}.
It can be used as a subbarrier at any $(t_0, z_0) \in \partial_0 \Omega$ such that 
 $t_0 \ge  0$ and $z_0 \in \partial \Omega$. The fact that it is a viscosity
supersolution follows from \cite{EGZ11, W12}. 

 2.2. Assume  $t_0 = 0$ and $z_0 \in \overline \Omega$.  
  Set for $t \in [0,T[$
  $$
\Gamma (t) := \int_0^t \gamma_+ (s) d s, \, \, \text{where} \, \, \gamma (t)  := 
- \inf \{F (t,z, \p_0 (z)) ; z \in \overline{\Omega}\}.
  $$
 Observe that $\Gamma$ is $C^1$ in $[0,T[$ and satisfies $\Gamma'(t) + F (t,z,\p_0 (z))\geq 0$ for all 
$(t,z) \in [0,T[ \times \Omega$. Thus
 $$
 V (t,z) := \p_0 (z) + C t + \Gamma (t),
 $$
 is a continuous function in $[0,T[ \times \overline{\Omega}$, $C^1$ in $t$ and  psh in $z$. Moreover for any $t \in [0,T[$,  it satisfies  
 $$ 
 (dd^c V_t)^n =  (dd^c \p_0) \leq e^C \mu \leq  e^{\partial_t V + F (t,\cdot,V_t)} \mu,
 $$ 
 in the pluripotential sense in $\Omega$. 
As above we infer that $V$ is a subsolution to the parabolic equation (\ref{eq:CD-PMAF}). 

 Since $V (0,z) = \p_0 (z) \geq \f_0 (z)- \e$, it follows that $V$ is an $\e$-superbarrier to the Cauchy problem (\ref{eq:CD-PMAF}) at any parabolic boundary point $(0,z_0) \in  \Omega_{T}$.
\end{proof}

Note that one cannot expect the existence of superbarriers when $\mu=0$ and $\f_0$ is not maximal.

\subsection{The Perron envelope}

We are now ready to show the existence of solutions to the Cauchy-Dirichlet problem for degenerate complex Monge-Amp\`ere flows:

 \begin{theo} \label{thm:existenceflot}
Assume $\mu >0$ or  $\mu=\mu(z)$ is independent of $t$ and $(\f_0,\mu)$ is admissible.
Then the Cauchy-Dirichlet problem for the parabolic complex Monge-Amp\`ere equation  
(\ref{eq:PMAF}) with Cauchy-Dirichlet condition (\ref{eq:CD-PMAF}) admits a unique viscosity solution $\f (t,z)$ in infinite time. 
 \end{theo}

 \begin{proof} 
It follows from Proposition \ref{prop:bar}
that there is at least a subsolution $\underline u$ and a supersolution $\overline v = 0$ to
 the Cauchy problem for the parabolic complex Monge-Amp\`ere equation (\ref{eq:PMAF}) with Cauchy-Dirichlet condition (\ref{eq:CD-PMAF}), 
which satisfy the inequality $\underline u  \leq  \overline v$ in $\R^+ \times \Omega$. 
We can thus consider the upper envelope $\f $ of those subsolutions 
$u$ that satisfy $\underline u \leq u \leq \overline v$ in $\R^+ \times \Omega$ as defined in \ref{eq:env}. 
 
 Fix $T > 0$ large and observe that the restriction of $\f^*$ to $\Omega_T$ is a subsolution 
to the parabolic complex Monge-Amp\`ere equation (\ref{eq:PMAF}), 
while the restriction of $\f_*$ to $\Omega_T$ is a supersolution to  the same parabolic complex Monge-Amp\`ere equation.  By Lemma~\ref{lem:Perron1}, they satisfy the inequalitiy (\ref{eq:Perron1}) and then 
 by semi-continuity there exists  $(t_0,x_0) \in (\{0\} \times  \overline{\Omega}) \cup ([0,T] \times \partial \Omega) $ such that
 $$
 \max_{(t,x) \in \overline{\Omega}_T} \{\f^* (t,x)  - \f_* (t,x)\} = \f^* (t_0,x_0)  - \f_* (t_0,x_0).
 $$ 

Fix $\e > 0$  arbitrary small. By Proposition~\ref{prop:bar}, there exists 
a continuous $\e-$subbarrier $U$ and an $\e-$superbarrier $V$ to the Cauchy-Dirichlet problem (\ref{eq:CD-PMAF}) in $\Omega_T$ at the parabolic boundary point $(t_0,x_0) \in \partial_0 \Omega_T$ such that 
$U (t_0,x_0) \geq \f_0 (x_0) - \e$ and  $V (t_0,x_0) \leq \f_0 (x_0) + \e$. 
Since $U_0 \leq \f_0 \leq V_0$ in $\Omega$, 
it follows from the comparison principle that $U \leq\f_* \leq \f \leq \f^* \leq V$ in $ [0,T]  \times \Omega$. Hence $ U = U_* \leq \f_*$ and $\f \leq V^* = V$ is $[0,T] \times \overline{\Omega}$. 
At the boundary point $(t_0,x_0)$ we have 
$$
\f_0 (x_0) - \e \leq U (t_0,x_0) \leq \f_* (t_0,x_0)\leq \f^* (t_0,x_0) \leq V (t_0,x_0) \leq \f_0 (x_0) + \e.
$$
We infer that for all $(t,x) \in  [0,T[ \times \Omega$, 
$$
 \f^* (t,x)  - \f_* (t,x) \leq \f^* (t_0,x_0)  - \f_* (t_0,x_0) \leq 2 \e. 
$$
Since $T > 0$ was arbitrary large, this implies that $\f^* \leq \f_*$ in $\R^+  \times \Omega$, 
hence  $\f^* = \f_*$ in $\R^+  \times \Omega$. 

The same reasoning as above shows that 
$\f (0,\cdot) = \f_0$ in $\overline{\Omega}$. This proves that $\f=\f^*$ is a continuous solution to the 
Cauchy-Dirichlet problem (\ref{eq:CD-PMAF}) in $\R^ 
+ \times \Omega$ with initial data $\f_0$.
 \end{proof}

\begin{rem}
When $\mu$ vanishes identically in a non empty open set $D \subset \Omega$ where $\f_0$ is not maximal
(in particular $(\f_0,\mu)$ is not admissible), then there is no viscosity solution to the above 
Cauchy-Dirichlet problem by Corollary~\ref{SuperMax}.
\end{rem}

\section{Long term behavior of the flows}

We assume in this last section that $F=F(z,r)$ is time independent.
It follows from Theorem \ref{thm:existenceflot} that the complex Monge-Amp\`ere flow
\begin{equation} \label{eq:CMAF2}
e^{\partial_t{\f}+  F(\cdot,\f)} \, \mu(z)-(dd^c \f_t)^n=0
\end{equation}
admits a unique solution for all times (i.e. makes sense in $\R^+  \times \Omega$) 
and for every
Cauchy-Dirichlet data $\f_0 \in {\mathcal C}^0({\partial \Omega}) \cap PSH(\Omega)$
such that $(\f_0,\mu)$ is admissible: we always assume such is the case in the sequel.

Our aim in this final section is to analyze, the asymptotic behavior of this flow when $t \rightarrow +\infty$. 
By analogy with the K\"ahler-Ricci flow, the model case is when 
$$
F(z,r)= h(z) + \al r, \; \; (t,z) \in \Omega \times \R.
$$
The situation is simple when $\al>0$ (negative curvature), more involved when $\al=0$
(Ricci flat case), often intractable when $\al <0$ (positive curvature). 

\subsection{Negative curvature}

We first make a strong assumption on $F$ (corresponding to the model case
$F(z,x)=\al x+h(z)$ with $\al>0$) so as to obtain a good control on the 
speed of convergence of
the flow, starting from {any} admissible initial data $\f_0$:

\begin{theo} \label{thm:conv<0}
Assume that the function $r \mapsto F(\cdot,r)-\al r$ is  increasing for some $\al>0$.
Then the complex Monge-Amp\` ere flow $\f_t$ starting at $\f_0$ uniformly converges, 
as $t \rightarrow +\infty$, to the solution $\p$
of the Dirichlet problem for the degenerate elliptic Monge-Amp\`ere equation
$$
(dd^c \p)^n=e^{F(z,\p)} \mu(z) \text{ in } \Omega,
\; \; \text{ with }
\p_{| \partial \Omega}=\f_0.
$$
More precisely
$$
|| \f_t-\p||_{L^{\infty}(\Omega)} \leq e^{-\al t} || \f_0-\p||_{L^{\infty}(\Omega)}
$$
\end{theo}

The existence of the solution $\p$ is well known in this case (see \cite{Ceg84}).

\begin{proof}
Consider 
$$
u (t,z) := e^{\al t} \f (t,z).
$$
Then $u$ is a solution to the  parabolic complex Monge-Amp\` ere equation
$$
e^{h(t)  \partial_t u + G (t,\cdot,u_t)} \mu = (dd^ c u_t)^ n, 
$$
where $ h(t) := e^{- \alpha t} $ and 
$$ 
G (t,z,r) := F (z,e^ {- \al t} r) - \alpha r e^ {- \al t}+n \al t.
$$

We let the reader check that
$ v (t,z) := e^{\alpha t} \psi (z)$ is a solution to the same parabolic complex Monge-Amp\` ere equation. 
Our hypothesis on $F$ implies that $r \longmapsto G (t,z,r)$ in non decreasing. We can 
thus apply the comparison principle (see Remark~\ref{GCP}), which yields the desired bound.
\end{proof}

\subsection{The general case}

We now show that the convergence holds in full generality, without any control on the 
speed of convergence:

\begin{theo} \label{thm:conv}
The complex Monge-Amp\` ere flow $\f_t$ starting at $\f_0$ uniformly converges, 
as $t \rightarrow +\infty$, to the solution $\p$
of the Dirichlet problem for the degenerate elliptic Monge-Amp\`ere equation
$$
(dd^c \p)^n=e^{F(z,\p)} \mu(z) \text{ in } \Omega,
\; \; \text{ with }
\p_{| \partial \Omega}=\f_0.
$$
\end{theo}

\begin{proof}
 We are going to use Theorem \ref{thm:conv<0}
 by considering the perturbed Monge-Amp\`ere flows associated to the functions $ F (z,r) + \e (r - c)$, where $\e > 0$ is small   and $c$ is a carefully chosen constant.

We first establish an upper bound. Set $M_0 := \sup_{\bar \Omega} \f_0$. Since the constant $M_0$ is a supersolution to the Monge-Amp\`ere flow associated to $(F,\mu)$ with boundary value $M_0$, it follows from the comparison principle that 
$$
\f (t,z) \leq M_0 \; \text{ in } \; \R^+ \times \Omega.  
$$

Fix  $\e > 0$ and set $ F^\e  (z,r) := F (z,r) + \e (r - M_0)$.
Let $\f^\e (t,z)$ be the solution of the complex
Monge-Amp\`ere flow associated to $(F^\e,\mu)$ with  Cauchy-Dirichlet data $\f^\e_0 = \f_{0}$ i.e. 
$$
 (dd^c \f^\e_t)^n=e^{\partial_t \f^\e + F (z,\f^\e) + \e (\f^\e - M_0)} \mu(z). \leqno (\star)_\e
$$

 Observe that $\f$ is a subsolution to the  flow $(\star)_\e$ since $\f \leq M_0$. 
The comparison principle therefore implies $\f \le \f^\e$ in $\R^+ \times \Omega$. 

Let $u^\e$ be the solution
of the degenerate elliptic Monge-Amp\`ere equation $(dd^c u^\e)^n= e^{F (z,u^\e) + \e (u^\e - M_0)} \mu(z)$ with Dirichlet data $u^\e|_{\partial \Omega} = \f_{0}|_{\partial \Omega}$ (\cite{Ceg84}).
It follows from the stability of the solutions to the Dirichlet problem 
for the complex Monge-Amp\`ere operator that $u^\e$ uniformly converges  to $u$ in $\Omega$  as 
$\e \to 0$ (see \cite{GKZ08}). 

Fix $\delta>0$ and choose $\e$ such that $u - \delta \leq  u^\e \le u +\delta$. 
It follows from Theorem \ref{thm:conv<0}
that $\lim_{t\to \infty} \f^\e_t (z)= u^\epsilon (z)$ uniformly in $\Omega$. Therefore there exists 
$T_\delta > 1$ so that for $t \geq T_\delta$ and $z \in \Omega$, $ \f_t(z) \le  u (z) + 2 \delta$.
This is the desired upper bound.

\smallskip

We now establish a lower bound.
Observe first that the family $(\f_t)$ is uniformly bounded from below. Indeed let $\rho$ be a strongly psh defining function for $\Omega$ and choose $B > 1$ such that 
$$
B^n (dd^c \rho)^n \geq e^{F (z,0)} \mu (z)
$$
 pointwise in $\Omega$. Since $\rho \leq 0$, the function $ \p (t,z) := B \rho (z)$ is a  subsolution to the parabolic Monge-Amp\`ere equation 
$ (dd^c \p_t)^n = e^{\partial_t \p + F (z,\p)} \mu (z) $. 
It therefore follows from the comparison principle that 
$$
B  \rho (z) - \f (t,z) \leq \max_{\bar \Omega} (B \rho - \f_0)_+
\text{ in }
\R^+ \times \Omega.
$$ 
Thus $\f$ is uniformly bounded from below by a constant $m_0$ in $\R^+ \times \Omega$. 

 We now consider the perturbed Monge-Amp\`ere flow  associated to $(F_\e,\mu)$ with  Cauchy-Dirichlet data $\f^\e_0 = \f_{0}$, where $ F_\e  (z,r) := F (z,r) + \e (r - m_0)$. 
Observe that $\f$ is a supersolution of this new perturbed flow since $\f \geq m_0$. 
Arguing as above shows the existence of $T_\delta' > 1$ such that 
$$
\f_t(z) \ge  u (z) - \delta
\text{ for }
t \geq T_\delta'
$$ 
and $z \in \Omega$. This proves that $\f_t \to u$  uniformly in $\Omega$.
\end{proof}


\begin{thebibliography}{99}


\bibitem[Ale39]{Ale39}  A.D.~Alexandrov: {\em Almost everywhere existence of the second order differential of a convex function and some properties of convex functions}. Leningrad. Univ. Ann. (Math. Ser.) {\bf 37} (1939), 3--35. (Russian)


 \bibitem[AHP12]{AHP12} B.~Avelin, L.~Hed, H.~Persson: { \em Approximation and Bounded Plurisubharmonic Exhaustion Functions Beyond Lipschitz Domains}. Preprint arXiv:1210.7105.


\bibitem[AFS08]{AFS08} D.~Azagra, D.~Ferrara; B.~Sanz: {\em Viscosity Solutions to second order partial differential equations on Riemannian manifolds}. J. Diff. Equations {\bf 245} (2008),
307–336.



 
 

 
 
 \bibitem[B09]{B09} R.~Berman: {\em Bergman kernels for weighted polynomials and weighted equilibrium measures of $\mathbb{C}^{n}$. } Indiana Univ. Math. J. 58 No. 4 (2009), 1921–1946
 
\bibitem[BD09]{BD09} R.~Berman,  J.-P.~Demailly: {\em Regularity of plurisubharmonic upper enveloppes in 
 big cohomology classes.}  Perspectives in analysis, geometry, and topology,  3966,
Progr. Math., 296, Birkh\"auser/Springer, New York, 2012. 

 \bibitem[BT76]{BT76} E.~Bedford, B.A.~Taylor: {\em The Dirichlet problem for a complex Monge-Amp\`ere equation.}  
   Invent. Math. {\bf 37}  (1976), no. 1, 1--44. 
   
\bibitem[BT82]{BT82} E.~Bedford, B.A.~Taylor: {\em  A new capacity for plurisubharmonic functions. } Acta Math.  {\bf 149}  (1982), 1--40.











    

  
  
  


 \bibitem[Bre59]{Bre59} H.J.~Bremermann: {\em On a generalized Dirichlet problem for plurisubharmonic functions and pseudo-convex domains. Characterization of Šilov boundaries.} Trans. Amer. Math. Soc. 91 (1959) 246–276.
  
  \bibitem[Car04]{Car04}  P.~Cardaliaguet:
{\em Solutions de viscosit\'e d'\'equations elliptiques et paraboloiques non lin\'eaires.} Notes de Cours, Universit\'e de Rennes, Janvier 2004. 
  
  
 \bibitem[Ceg84]{Ceg84} U.~Cegrell: {\em On the Dirichlet problem for the complex Monge-Amp\`ere operator.} Math. Z. {\bf 185} (1984), no. 2, 247--251. 

 

 
 
\bibitem[CC95]{CC95} L.A.~Caffarelli, X.~Cabr\'e: {\em Fully Nonlinear Elliptic Equations}. American Mathematical Society Colloqium publications, Vol. 43,  1995.   
  

 \bibitem[CIL92]{CIL92} M.~Crandall, H.~Ishii, , P.L.~Lions: {\em User's guide to viscosity solutions of second order 
partial differential equations} Bull. Amer. Math. Soc. {\bf 27} (1992), 1-67. 




 


 
 \bibitem[DI04]{DI04}  J.~Droniou, C.~ Imbert:  {\em Solutions et solutions variationnelles pour EDP  non lineaire}, Cours polycopi\'e (2004), Universit\'e de Montpellier.

 
  
  
 \bibitem[EGZ11]{EGZ11}  P.~Eyssidieux, V.~ Guedj, A.~ Zeriahi: {\em Viscosity solutions to Degenerate Complex Monge-Amp\`ere Equations.} 
  Comm. Pure Appl. Math.  {\bf 64}  (2011),  no. 8, 10591094. 
 
 \bibitem[EGZ13]{EGZ13}  P.~Eyssidieux, V.~ Guedj, A.~ Zeriahi: {\em Continuous approximation of quasi-plurisubharmonic functions}, Preprint 2013, arXiv:1311.2866.

 \bibitem[EGZ14]{EGZ14}  P.~Eyssidieux, V.~ Guedj, A.~ Zeriahi: {\em Weak solutions to degenerate complex Monge-Amp\`ere flows II.} Preprint 2014.
  
\bibitem[FW89]{FW89}  J.-E.Fornaess, J.Wiegerinck:  {\em Approximation of plurisubharmonic functions.}
Ark. Mat. {\bf 27} (1989), no. 2, 257-272.

  
 \bibitem[Gav77]{Gav} B.~Gaveau: {\em M\'ethodes de contr\^ole optimal en analyse complexe.I. R\'esolution d'\'equations de Monge-Amp\`ere.} J. Funct. Anal. {\bf 25} (1977), 391-411. 


\bibitem[GKZ08]{GKZ08} V.~Guedj, S.~Kolodziej, A.~Zeriahi: H\"older continuous solutions to Monge-Amp\`ere equations.
 Bulletin of London Math. Soc {\bf 40} (2008), 1070-1080.


   
 
 
 
 
 \bibitem[HL09]{HL09}  F.R.~Harvey, H.B.~Lawson: {\em Dirichlet duality and the nonlinear Dirichlet problem.}  Comm. Pure Appl. Math.  62  (2009),  no. 3, 396--443.
 
 \bibitem[H\"or94]{Hor94} L.~H\"ormander: {\em Notions of convexity}, Progress in Math., Birkh\"auser (1994).

 
 \bibitem[Ish89]{Ish89} H.~Ishii:  {\em On uniqueness and existence of viscosity solutions of fully nonlinear 
 second-order elliptic PDEs},  Comm. Pure Appl. Math.  {\bf 42}  (1989),  no. 1, 15--45.  

 \bibitem[IL90]{IL90} H.~Ishii, P.L.~Lions:{\em Viscosity solutions of fully nonlinear second-order elliptic
 partial differential equations}, Journ. Diff. Equations {\bf 83} (1990), 26--78.

 \bibitem[IS13]{IS12}  C.~Imbert, L. Sylvestre: {\em Introduction to fully nonlinear parabolic equations}.
{\it An introduction to the K\"ahler-Ricci flow.} Lecture Notes in Math., {\bf 2086}, Springer, Heidelberg, 2013.

 
 \bibitem[Jen88]{Jen88} R.~Jensen: {\em The maximum principle for viscosity solutions of fully nonlinear second-order partial differential  equations}, Arch. Rat. Mech. Anal. {\bf 101} (1988), 1-27.






 
 


 
 
 

\bibitem [Sad81]{Sad81} A.~Sadullaev: {\em Plurisubharmonic measures and capacities on complex manifolds.} (Russian) Uspekhi Mat. Nauk 36 (1981), no. 4 (220), 53–105, 247.
 
\bibitem [Sib87]{Sib87} N.~Sibony: { \em Une classe de domaines pseudoconvexes.}  
 Duke Math. J. {\bf 55} (1987), no. 2, 299-319. 

\bibitem[ST12]{ST1} J.~Song, G.~Tian: {\em Canonical measures and K\"ahler Ricci flow.} J. Amer. Math. Soc. {\bf 25} (2012), 303--353. 
 
\bibitem[ST09]{ST2} J.~Song, G.~Tian: {\em The K\"ahler-Ricci flow through singularities.} Preprint arXiv:0909.4898 (2009).

 \bibitem[W12]{W12} Y.~Wang: {\em A Viscosity Approach to the Dirichlet Problem for Complex Monge-Amp\`ere Equations. } Math. Z. {\bf 272} (2012), no. 1-2, 497--513. 
    
  
 
 
 
 
\bibitem[Ze13]{Ze13} A.~Zeriahi: {\em A viscosity approach to degenerate complex Monge-Amp\`ere equation}.  Ann. Fac. Sci. Toulouse Math. (6) 22 (2013), no. 4, 843--913.

 
\end{thebibliography}
 \end{document}